\documentclass{article}

\usepackage{amsfonts,amsbsy,amsmath,amssymb,amsthm,makeidx,latexsym,epsfig}
\usepackage{verbatim}
\usepackage{color}

\usepackage{mathrsfs}  

\newtheorem{defn}{Definition}[section]
\newtheorem{lemma}[defn]{Lemma}
\newtheorem{proposition}[defn]{Proposition}
\newtheorem{ex}[defn]{Example}
\newtheorem{thm}[defn]{Theorem}
\newtheorem{prop}[defn]{Proposition}
\newtheorem{cor}[defn]{Corollary}
\newtheorem{rem}[defn]{Remark}

\newtheorem{fig}[defn]{Figure}

\newcommand{\ltr}{ L^2(\mathbb R) }

\newcommand{\mn}{\mathbb N}
\newcommand{\mr}{\mathbb R}

\newcommand{\mz}{\mathbb Z}

\newcommand{\mt}{ E_{mb}T_{na}g}
\def\bp{{\noindent\bf Proof. \ }}
\def\ep{\hfill$\square$\par\bigskip}

\def\bqs{\begin{equation}}
\def\eqs{\tag*{$\square$}\end{equation}\par\bigskip}

\def\la{\langle}
\def\ra{\rangle}
\def\ga{\gamma}

\def\sukz{\sum\limits_{k\in \mz}}

\def\Span{\text{span}}

\def\bop{\begin{op}\rm}
\def\eop{\end{op}}

\def\bee{\begin{eqnarray}}
\def\ene{\end{eqnarray}}
\def\bes{\begin{eqnarray*}}
\def\ens{\end{eqnarray*}}
\def\bei{\begin{itemize}}
\def\eni{\end{itemize}}
\def\bt{\begin{thm}}
\def\et{\end{thm}}
\def\bc{\begin{cor}}
\def\ec{\end{cor}}
\def\bpr{\begin{prop}}
\def\epr{\end{prop}}
\def\bl{\begin{lemma}}
\def\el{\end{lemma}}
\def\bd{\begin{defn}}
\def\ed{\end{defn}}
\def\bex{\begin{ex}}
\def\enx{\end{ex}}
\def\bfi{\begin{fig}}
\def\efi{\end{fig}}

\def\mt{{\mathbb T}}

\newcommand{\R}{\mathbb{R}}
\newcommand{\Z}{\mathbb{Z}}
\newcommand{\C}{\mathbb{C}}
\newcommand{\N}{\mathbb{N}}
\newcommand{\T}{\mathbb{T}}

\newcommand{\be}{\begin{equation}}
\newcommand{\ee}{\end{equation}}
\newcommand{\re}{\textrm{Re\,}}
\newcommand{\im}{\textrm{Im\,}}
\newcommand{\mydot}{\,\cdot\,}
\newcommand{\ml}{\vskip 6pt\noindent}
\DeclareMathOperator{\Arg}{Arg}
\DeclareMathOperator{\clos}{clos}
\newcommand{\veps}{\varepsilon}
\newcommand{\sP}{\mathscr{P}}
\newcommand{\inn}[2]{{\langle #1,#2 \rangle}}

\usepackage{hyperref}

\title{Fractional and Complex Pseudo-Splines and the Construction of Parseval Frames}

\date{\today}

\author{Peter Massopust\footnote{Research partially supported by DFG Grant MA 5801/2-1}, Brigitte Forster\footnote{Research partially supported by DFG Grant FO 792/2-1}, Ole Christensen}

\begin{document}

\maketitle

\begin{abstract} Pseudo-splines of integer order $(m,\ell)$ were introduced by Daubechies, Han, Ron, and Shen
as a family which allows interpolation between the classical B-splines and the Daubechies' scaling
functions. The purpose of this paper is to generalize the pseudo-splines to fractional and complex orders $(z, \ell)$ with $\alpha:=\re z \geq 1$. This allows increased flexibility in regard to smoothness: instead of working with a discrete family of functions from $C^m$, $m\in \N_0$, one uses a \emph{continuous} family of functions belonging to the H\"older spaces $C^{\alpha-1}$. The presence of the imaginary part of $z$ allows for direct utilization in complex transform techniques for signal and image analyses.  We also show that in analogue to the integer case, the generalized pseudo-splines lead to constructions of Parseval wavelet frames via the unitary extension principle. The regularity and approximation order of this new class of generalized splines is also discussed.
\vskip 6pt\noindent
\textit{MSC Classification (2010)}: 42C15, 42C40, 65D07
\end{abstract}
\section{Introduction}
Pseudo-splines of integer order $(m,\ell)$ were introduced in the paper \cite{DRoSh3} by Daube\-chies, Han, Ron, and Shen and studied further  in \cite{DongShen,DongShen2}. The motivation was to construct families of refinable functions which allow the interpolation between the classical B-splines and the Daubechies' scaling functions. Their focus was on the construction of  framelets for $L^2(\R)$.

The idea of this paper is to generalize the pseudo-splines to fractional and complex orders $(z, \ell)$ with $\alpha:=\re z \geq 1$. This allows increased flexibility in regard to smoothness: instead of working with a discrete family of functions from $C^{m-1}$, $m\in \N$, ones uses a \emph{continuous} family of functions belonging to the H\"older spaces $C^{\alpha-1}$. The presence of the imaginary part of $z$ allows for direct utilization in complex transform techniques for signal and image analyses. The usefulness of such complex-valued transforms is discussed in, for instance, \cite{forster14}: Real-valued transforms can only provide a symmetric spectrum and are therefore unable to separate positive and negative frequency bands. Moreover, in applications such as digital holography, real-valued transforms cannot be applied to retrieve the phase $\theta(x)$ of a complex-valued wave $b(x)e^{i\theta(x)}$. Here, complex-valued transforms, bases and frames are indispensably required.

In \cite{DongShen} it is mentioned that pseudo-splines of order $z = m \in \N$ with $\ell=0$ are B-splines. The same is true for the pseudo-splines of fractional and complex order. In fact, for real $\alpha = z = \re z \geq 1$, the pseudo-splines with parameter $\ell=0$ are the symmetric fractional $B$-splines $\beta^{2\alpha-1}_{\ast}$ defined in \cite{unserblu00}. For complex $z$ with $\re z \geq 1$ and $\ell=0$ they correspond to the complex B-splines $\beta^{2z-1}$ as they are considered in \cite{forster06}.

In this paper, we are interested in the construction of Parseval wavelet frames of the form $\{D^jT_k \psi_l : j,k\in \Z;\, l = 0,1,\ldots, n\}$, where $D$ denotes the unitary dilation operator $(D f)(x) := \sqrt{2} f(2x),$  $T_k$ the translation operator $(T_k f)(x) := f(x-k)$, $k\in \Z$, and $\{\psi_1, \dots, \psi_n\}$ is a collection of functions in $\ltr.$ That is, we want to construct the functions $\psi_l$ such that
\bes
\sum_{l=1}^n \sum_{j,k\in \mz} | \la f, D^jT_k\psi_l\ra|^2= \lVert f \rVert^2, \quad \forall f\in \ltr,
\ens
or, equivalently,
\bes f= \sum_{l=1}^n\sum_{j,k\in \mz}  \la f, D^jT_k\psi_l\ra D^jT_k\psi_l, \quad \forall f\in \ltr.
\ens

In contrast to the classical Schoenberg polynomial splines of even degree $m$, which allow the  construction of Parseval wavelet frames for $L^2(\R)$ via the unitary or oblique extension principle \cite{RoSh4}, the fractional and complex B-splines \cite{forster06,unserblu00} do not yield such frames: The trigonometric identity $1=(\sin^2 \pi\gamma + \cos^2 \pi\gamma)^m$, $\gamma\in (-\frac12,\frac12)$, which is used to construct Parseval wavelet frames in \cite{RoSh2}, is a finite sum involving powers of $\sin$ and $\cos$ whereas in the case when $z\in\C\setminus\N$, $\re z \geq 1$, the corresponding expression $1=(\sin^2 \pi\gamma + \cos^2 \pi\gamma)^z$ becomes an infinite binomial series which does not converge on all of $(-\frac12,\frac12)$. Indeed,
\[
\sum\limits_{k=0}^\infty \binom{z}{k} (\sin^2 \pi\gamma)^k\,(\cos^2 \pi\gamma)^{z-k} \;\;\text{converges for $\gamma\in (-\tfrac14,\tfrac14)$},
\]
and
\[
\sum\limits_{k=0}^\infty \binom{z}{k} (\cos^2 \pi\gamma)^k\,(\sin^2 \pi\gamma)^{z-k} \;\;\text{converges for $\gamma\in (-\tfrac12,\tfrac12)\setminus (-\tfrac14,\tfrac14)$}.
\]
This splitting of the binomial series into these two parts with disjoint regions of convergence causes the construction to fail. Here, however, we show that in analogue to the integer case, the generalized fractional and complex pseudo-splines lead to constructions of Parseval wavelet frames via the unitary extension principle.

The structure of the paper is as follows. In Section 2, we introduce pseudo-splines of fractional and complex order and show that they generate refinable functions via the cascade algorithm. Lowpass properties of complex pseudo-splines are investigated in Section 3. In Section 4, we use the unitary extension principle to construct Parseval wavelet frames from these functions. Section 5 discusses the regularity and approximation order for this new class of generalized splines.

Throughout the paper we will use the notation $\N := \{1, 2, 3, \ldots\}$, $\N_0 := \N\cup\{0\}$, $\Z_0^- := -\N_0$. The torus is $\T := [-\frac12, \frac12] \cong\R/\Z$. The floor function $\lfloor\mydot\rfloor:\R\to \Z$ is given by $r\mapsto \max\{n\in \Z : n\leq r\}$ and the ceiling function $\lceil\mydot\rceil:\R\to \Z$ is $r\mapsto \min\{n\in \Z : n\geq r\}$.
The class of 1-periodic
functions on $\R$ whose restriction to $(-\frac12, \frac12)$
belongs to $L^2(-\frac12, \frac12)$ is denoted by
$L^2({\mathbb T})$. Similarly, $L^\infty(\mt)$ consists of the
bounded measurable 1-periodic functions on $\mr$.

\section{Pseudo-Splines of Fractional and Complex Order}
In this section, we generalize the pseudo-splines of type II introduced in \cite{DRoSh3} to fractional and complex order and show that they generate refinable functions via the cascade algorithm.

To this end, we consider the filter
\be\label{eq1}
H_0 (\gamma) := H_0^{(z, \ell)}(\gamma) := (\cos^2\pi\gamma)^z\,\sum\limits_{k=0}^\ell \binom{z + \ell}{k} (\sin^2\pi\gamma)^k\,(\cos^2\pi\gamma)^{\ell-k}, \, \ga \in \mr,
\ee
where $z\in \C$ with $\alpha:= \re z \geq 1$ and $0\leq\ell\leq \lfloor \alpha - \frac12\rfloor$. The binomial coefficient is short-hand notation for
\[
\binom{z + \ell}{k} := \frac{\Gamma (z+\ell+1)}{\Gamma (k+1) \Gamma (z+\ell - k+1)},
\]
where $\Gamma: \C\setminus\Z^-_0\to\C$ denotes the Euler Gamma function. We note that since $\cos^2\pi\gamma\geq 0$, for all $\gamma\in \T$, and $\alpha \geq 1$, the complex-valued expression $(\cos^2\pi\gamma)^z$ is well-defined and holomorphic on $\C_{\geq 1} := \{\zeta\in \C : \re\zeta \geq 1\}$.

The paper \cite{DRoSh3} treats the case where $z$ and $\ell$ are nonnegative
integers with $\ell < z.$  We will show that under the above assumptions
the function $H_0$ is indeed the filter associated with a refinable function,
which generalizes the pseudo-splines to complex order.
\ml
For the remainder of this paper, the pair $(z,\ell)\in \C\times \N_0$ is assumed to be chosen according to the above restrictions and is held fixed.
\ml
Before we introduce the pseudo-splines we need to prove certain inequalities related to the filter $H_0,$ see Proposition \ref{9201}. We need the lemma below; in the integer case $z\in \N$ it is proved in \cite[Lemmata 2.1 and 2.2]{DongShen}.

\begin{lemma} \label{60116a}
Let $z\in \C_{\geq 1}$, let $\alpha := \re z$, let $\ell\in\N_0$ and define $\C$-valued polynomials $p(x) := p_{z,\ell}(x) := \sum\limits_{k=0}^\ell \binom{z+\ell}{k} x^k (1-x)^{\ell-k}$ and $q(x) := (1-x)^z p(x)$, for $x\in [0,1]$. Then $p$ can alternatively be written in the form
\be\label{p}
p(x) = \sum\limits_{k=0}^\ell \binom{z-1+k}{k} x^k,
\ee
and the derivative $q'$ of $q$ with respect to $x$ is given by
\be\label{q}
q' (x) = -(z+\ell)\binom{z-1+\ell}{\ell} x^\ell (1-x)^{z-1}.
\ee
\end{lemma}

\begin{proof}
The proof follows the ideas in \cite{DongShen} using the following two binomial identities that are also valid for complex $z$:
\begin{align}\label{iden}
(k+1)\binom{z+k}{k+1} & = (k+1)\,\frac{\Gamma (z+k+1)}{\Gamma (k+2)\Gamma (z)} = (k+1)\,\frac{(z+k)\Gamma (z+k)}{(k+1)\Gamma (k+1)\Gamma (z)}\nonumber\\
& = (z+k)\frac{\Gamma (z+k)}{\Gamma (k+1)\Gamma (z)} = (z+k) \binom{z-1+k}{k}
\end{align}
and Pascal's Identity
\[
\binom{z+1}{k} = \binom{z}{k} + \binom{z}{k-1}, \quad j\in \N,
\]
which is proved similar to the first identity.

The alternative representation of $p$ and the expression for the derivative of $q$ are shown as in \cite[Lemma 2.1, respectively, Lemma 2.2]{DongShen} replacing the natural integer $m$ by the complex number $z$.
\end{proof}

The following technical result will play a crucial role in the construction of
the refinable function associated with the filter $H_0,$ and also for the
construction of Parseval
frames based on pseudo-splines, see Corollary \ref{9603}.

\begin{proposition} \label{9201} Let $z\in \C_{\geq 1}$ and let $\ell = 0,1,\ldots, \lfloor\alpha-\frac12\rfloor$.
Then
\be\label{eq4}
0 < \vartheta \leq |H_0 (\gamma)|^2 + \left|H_0 \left(\gamma + \tfrac12\right)\right|^2 \leq 1, \quad\forall \gamma\in \T,
\ee
and some constant $\vartheta = \vartheta (z,\ell)$.
\end{proposition}

\begin{proof}
Set $x := \sin^2\pi\gamma\in [0,1]$. With a slight abuse of the
 notation $H_0 (\gamma)$ can then be written in the form
\[
H_0 (x)  = (1-x)^z p(x),
\]
where $p(x) := p_\ell(x) := \sum\limits_{k=0}^\ell \binom{z+\ell}{k} x^k (1-x)^{\ell-k}.$
Let $q(x) := (1-x)^z p(x)$ be as in Lemma \ref{60116a}.
To establish the validity of the inequalities, \eqref{eq4} is now equivalent to showing that
\[
s(x) := q(x)\overline{q(x)} + q(1-x)\overline{q(1-x)}
\]
is bounded above by one and below by some $\vartheta > 0$. To this end, consider the derivative of $s$:
\be\label{sprime}
s' (x) = q' (x)\overline{q(x)} + q(x) \overline{q'(x)} - q' (1-x)\overline{q(1-x)} - q(1-x) \overline{q'(1-x)}.
\ee
Using the expression for $q(x)$ and its derivative \eqref{q}, we obtain the following formulas for the four products in \eqref{sprime}:
\begin{align*}
q' (x)\overline{q(x)} &= -(z+\ell)\,\binom{z-1+\ell}{\ell}\,\sum\limits_{k=0}^\ell \overline{\binom{z-1+k}{k}} x^{\ell+k} (1-x)^{2\alpha -1},\\
q(x) \overline{q'(x)} &= \overline{q' (x)\overline{q(x)}},\\
q' (1-x)\overline{q(1-x)} & = -(z+\ell)\,\binom{z-1+\ell}{\ell}\,\sum\limits_{k=0}^\ell \overline{\binom{z-1+k}{k}} (1-x)^{\ell+k} x^{2\alpha -1},\\
q(1-x) \overline{q'(1-x)} & = \overline{q' (1-x)\overline{q(1-x)}}.
\end{align*}
Setting $A_k := A_k(z, \ell) := (z+\ell)\,\binom{z-1+\ell}{\ell} \overline{\binom{z-1+k}{k}}$, yields
\begin{align*}
q' (x)\overline{q(x)} + q(x) \overline{q'(x)} = - \sum\limits_{k=0}^\ell (2 \re A_k)\, x^{\ell+k} (1-x)^{2\alpha -1}
\end{align*}
and, similarly,
\begin{align*}
q' (1-x)\overline{q(1-x)} + q(1-x) \overline{q'(1-x)} = - \sum\limits_{k=0}^\ell (2 \re A_k)\, (1-x)^{\ell+k} x^{2\alpha -1}.
\end{align*}
Therefore,
\begin{align*}
s'(x) &= \sum\limits_{k=0}^\ell (2 \re A_k)\,\left[(1-x)^{\ell+k} x^{2\alpha -1} - x^{\ell+k} (1-x)^{2\alpha -1}\right].
\end{align*}
Hence, the only stationary points of $s$ are given by $x \in \{0, \frac12, 1\}$. Setting $x := \frac12+\veps$, $0 < \veps < \frac12$, the expression in brackets produces three cases.
\ml
Case I: If $\ell+k < 2\alpha - 1$ then
\[
(\tfrac12 - \veps)^{\ell+k}(\tfrac12 + \veps)^{\ell+k}\left[(\tfrac12 + \veps)^{2\alpha - 1 - \ell-k}-(\tfrac12 - \veps)^{2\alpha-1-\ell-k}\right]\begin{cases} > 0, & \veps > 0;\\ < 0, & \veps < 0.
\end{cases}
\]
Case II: If $\ell+k > 2\alpha - 1$ then
\[
(\tfrac12 - \veps)^{2\alpha - 1}(\tfrac12 + \veps)^{2\alpha - 1}\left[(\tfrac12 - \veps)^{\ell + k - 2\alpha + 1}-(\tfrac12 + \veps)^{\ell + k - 2\alpha + 1}\right]\begin{cases} < 0, & \veps > 0;\\ > 0, & \veps < 0.
\end{cases}
\]
Case III: If $\ell+k = 2\alpha - 1$ then there are only finitely many integers $k\geq 0$ and $0\leq k\leq \ell$ for which $(1-x)^{\ell+k} x^{2\alpha -1} - x^{\ell+k} (1-x)^{2\alpha -1}$ is identically equal to zero. The value of $s'$ is therefore determined by Cases I and II.

Thus, the stationary point $x=\frac12$ yields the unique minimum for $s$ only if $0\leq k \leq \ell \leq \lfloor\alpha - \frac12\rfloor\leq \alpha -\frac12$ holds.

In particular, $s(0) = s(1) = 1$ and
\begin{align*}
s \left(\frac12\right) &= 2 q\left(\frac12\right) \overline{q\left(\frac12\right)}\\
& = 2 \left(\frac12\right)^z \sum\limits_{k=0}^\ell \binom{z+\ell}{k} \left(\frac12\right)^\ell\cdot \overline{\left(\frac12\right)^z \sum\limits_{k=0}^\ell \binom{z+\ell}{k} \left(\frac12\right)^\ell}\\
& = 2^{1-2\alpha-2\ell} \left\lvert\sum\limits_{k=0}^\ell \binom{z+\ell}{k}\right\rvert^2.
\end{align*}
This proves the result with $\vartheta := \vartheta (z, \ell) := s(\frac12)$.
\end{proof}

\begin{figure}[h!]
\centering
\includegraphics[height = 4cm, width = 6cm]{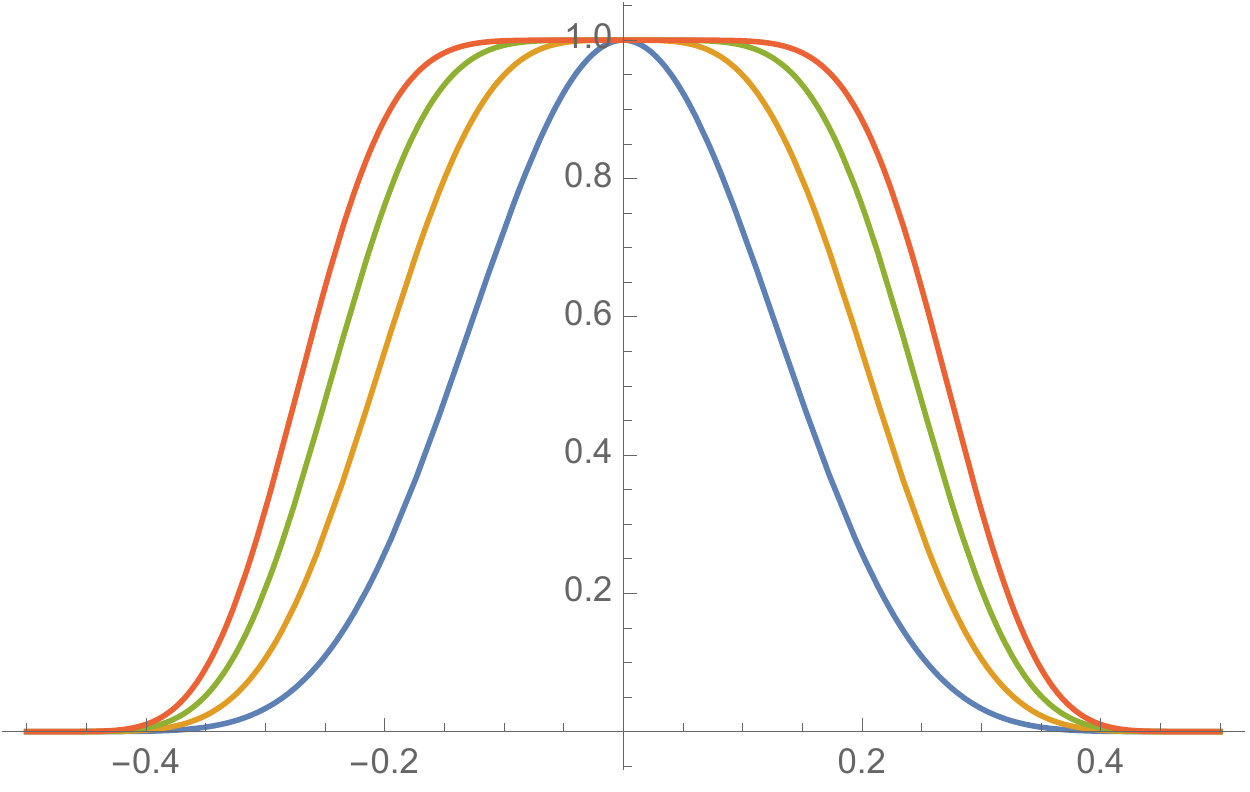}\\[2ex]
\includegraphics[height = 4cm, width = 6cm]{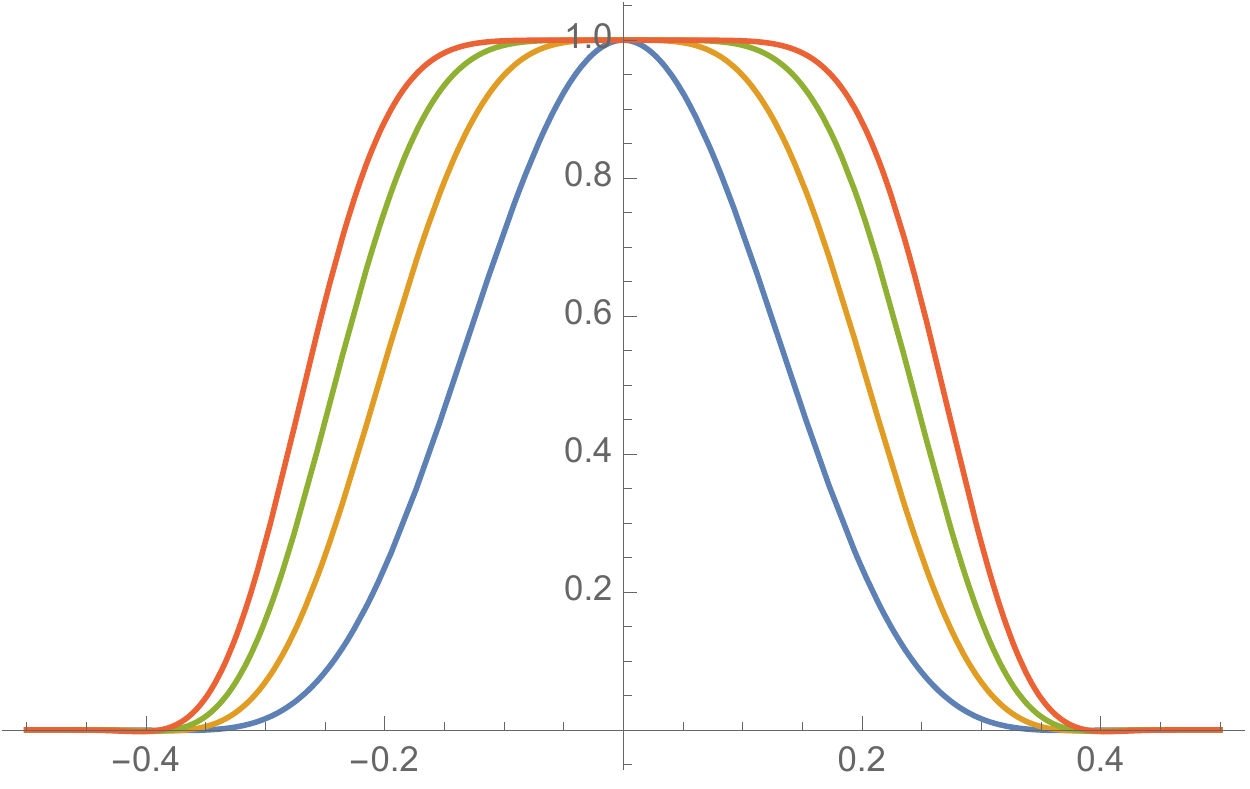}
\includegraphics[height = 4cm, width = 6cm]{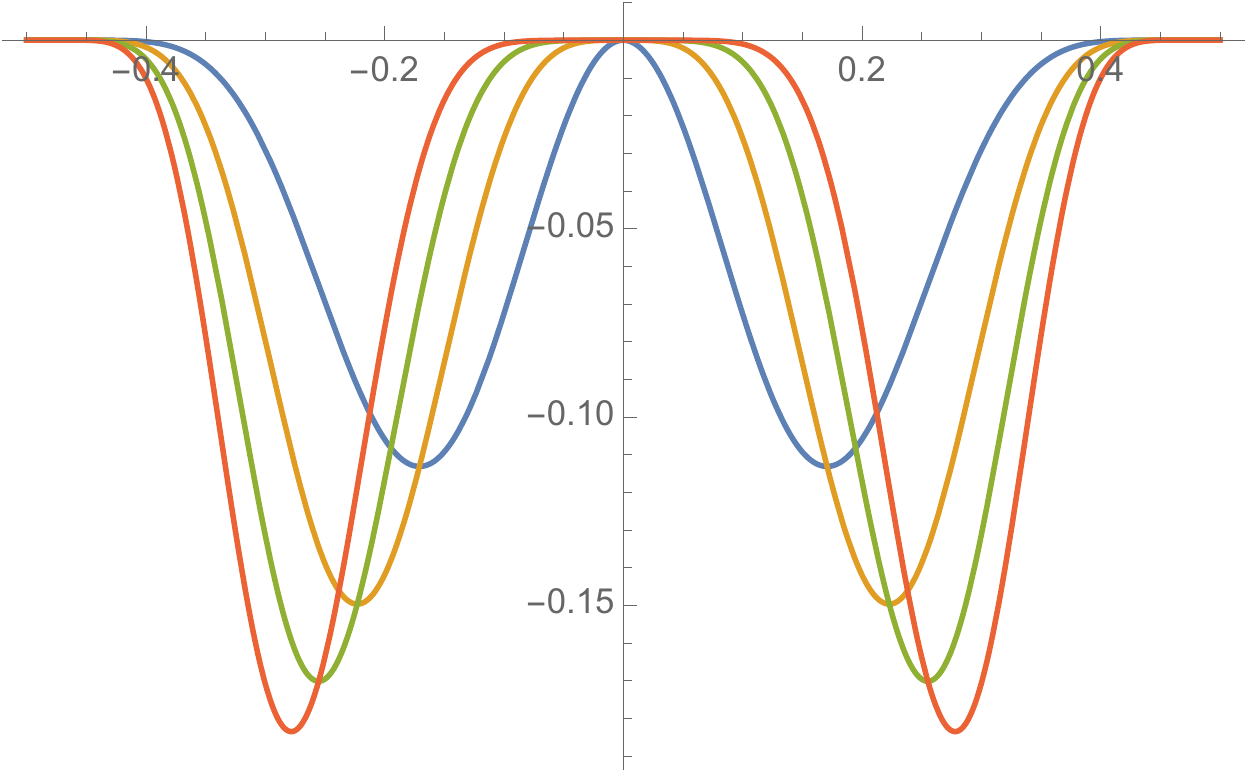}
\caption{The filter $H^{(3.2+i,\ell)}$ for $\ell= 0,1,2,3$ (inner to outer curves). The upper image shows $|H^{(3.2+ i,\ell)}|$, the lower left image $\re H^{(3.2+ i,\ell)}$, and the lower right image $\im H^{(3.2+ i,\ell)}$. The pseudospline parameter $\ell$ allows for a tuning of the width of the scaling filter. Notice that the imaginary part acts as an added bandpass filter.}
\end{figure}

In Theorem \ref{60131a} we will show that there indeed exists a refinable function $\varphi \in \ltr$ associated with the  filter $H_{0}\in L^{2} (\T)$ in \eqref{eq1}, i.e., a function such that for its Fourier transform holds:
\bee \label{60131b}
\widehat{\varphi}(\gamma) = H_0\left(\frac{\gamma}{2}\right) \widehat{\varphi}\left(\frac{\gamma}{2}\right),\qquad \gamma\in\R.
\ene
Here, we consider the Fourier transform with the normalization
\[
\widehat{\varphi}(\gamma) = \int_{\R} \varphi(t) e^{-2 \pi i t \gamma} \, dt,\quad \gamma \in \R.
\]

As in classical wavelet analysis we identify an appropriate function $\varphi$ via the cascade algorithm. Note that
iteration of \eqref{60131b}  formally yields that
\begin{equation}
\widehat{\varphi}(\gamma) = \prod\limits_{m=1}^\infty H_{0}(2^{-m}\gamma) \widehat{\varphi}(0), \quad \ga \in \mr.
\label{eq infinite product}
\end{equation}
Motivated by this we define the functions $\varphi_m, \, m\in\N_0,$ via

\begin{align}
\widehat{\varphi}_{0}(\gamma) & := \chi_{[-\frac{1}{2},\frac{1}{2}]}(\gamma),
\nonumber\\
\widehat{\varphi}_{m}(\gamma) & := \chi_{[-2^{m-1}, 2^{m-1}]} (\gamma) \prod\limits_{j=1}^{m} H_{0}(2^{-j}\gamma).
\label{eq Cascade Algorithm}
\end{align}

We will show that $\widehat{\varphi}_m$ converges to a function
in $\ltr$, to be denoted by $\widehat{\varphi},$ and that the inverse Fourier
transform $\varphi$ satisfies \eqref{60131b}. In order to do so, we need
the following general result, which  exists in many variants  in the literature, see e.g. \cite{DongShen2,lmr-englisch,woj97}.
It describes the exact interpretation of the convergence
of the infinite product in \eqref{eq infinite product}. Note that in this result
$H_0$ denotes \emph{any periodic} function, not necessarily the filter in \eqref{eq1}; in
contrast to the versions in the literature we do not assume that  $H_0$ is a trigonometric polynomial.

\begin{lemma}
\label{L L2 Konvergenz infinite product}
Let $H_{0}(\gamma)$ be a $1$-periodic real or complex function satisfying the following  conditions:

\begin{itemize}
\item [1)] $H_{0}(0)=1$.
\item [2)]  There exist  a constant $C>0$ and an exponent $\varepsilon > 0$ such that
$$|H_{0}(\gamma) - 1| \leq C \cdot |\gamma|^{\varepsilon}, \quad \mbox{for all } \gamma\in\R.$$
\item[3)] There exists a positive constant $\vartheta$ such that $$0 < \vartheta \leq |H_{0}(\gamma)|^{2} + |H_{0}(\gamma+\tfrac{1}{2})|^{2} \leq 1, \quad \mbox{for all } \gamma\in\R.  $$ 
\end{itemize}
Then the sequence $\widehat{\varphi}_{m}, \, m\in \mn$, converges pointwise and uniformly on compact subsets. The pointwise limit $\widehat{\varphi}$ belongs to $L^{2}(\R)$ and
$$
\varphi_{m} \to \varphi \quad \mbox{in\ } L^{2}(\R), \ \mbox{as } m\to \infty.
$$ Furthermore, the function $\varphi$ satisfies the refinement equation \eqref{60131b}.
\end{lemma}

\begin{proof}  It is well-known (see \cite{DongShen2,lmr-englisch}) that the assumptions 1) and 2)
imply that the sequence $\widehat{\varphi}_{m}, \, m\in \mn$, converges pointwise and uniformly on compact subsets.
We now show that the sequence  $\Vert\widehat{\varphi}_{m}\Vert_2$, $m\in\N_{0}$, is bounded, and that the sequence of functions $\{\widehat{\varphi}_{m}\}$ converges to $\widehat{\varphi} \in L^{2}(\R)$ as $m\to\infty$.
Plancherel's theorem then implies convergence in the time domain.

We start by splitting up $\widehat{\varphi}_{m+1}$ into a first factor consisting of the scaled filter $H_{0}$ and a second factor---a $2^{m}$-periodic filter product:

\begin{align*}
\Vert \widehat{\varphi}_{m+1}\Vert_{2}^{2} & = \int_{-2^{m}}^{2^{m}} |H_{0}(2^{-m-1}\gamma)|^{2} \left| \prod\limits_{j=1}^{m} H_{0}(2^{-j}\gamma)\right|^{2}\, d\gamma\\
& = \int_{0}^{2^{m}} \left(|H_{0}(2^{-m-1}\gamma)|^{2} + |H_{0}(2^{-m-1}(\gamma- 2^{m}))|^{2}\right)\left| \prod\limits_{j=1}^{m} H_{0}(2^{-j}\gamma)\right|^{2}\, d\gamma\\
& \leq \int_{0}^{2^{m}} \left| \prod\limits_{j=1}^{m} H_{0}(2^{-j}\gamma)\right|^{2}\, d\gamma
\quad \leq \quad \Vert \widehat{\varphi}_{m}\Vert_{2}^{2}.
\end{align*}

Since $\Vert \widehat{\varphi}_{0}\Vert_{2} = 1$, we find that the sequence of norms $\Vert \widehat{\varphi}_{m}\Vert_2$, $m \in \N_{0}$ is bounded by $1$. The pointwise convergence  together with the positivity of the squares $|\widehat{\varphi}_{m}(\gamma)|^{2}$ for all $m \in \N_{0}$ yield via Fatou's Lemma that the limit function $\widehat{\varphi}$ is square-integrable:
$$
\Vert \widehat{\varphi}\Vert_{2}^{2}\; \leq\; \limsup_{m\to\infty} \Vert \widehat{\varphi}_{m}\Vert_{2}^{2}\; \leq\; 1.
$$
It remains to prove the convergence of $\widehat{\varphi}_{m}$, $m\in\N_{0}$ to $\widehat{\varphi}$ in $L^{2}(\mr)$. This, however, follows directly from the proof of  \cite[Theorem 2.4.7]{lmr-englisch} and the comments given therein. Note that Cohen's criterion, which is required in the proof, is satisfied because of condition 2) and the inverse triangular inequality.
\end{proof}

\begin{thm} \label{60131a}
The filter in \eqref{eq1}, i.e.,
\[
H_0 (\gamma) := H_0^{(z, \ell)}(\gamma) := (\cos^2\pi\gamma)^z\,\sum\limits_{k=0}^\ell \binom{z + \ell}{k} (\sin^2\pi\gamma)^k\,(\cos^2\pi\gamma)^{\ell-k},
\]
where $z\in \C$ with $\alpha:= \re z \geq 1$ and $\ell = 0, 1, \ldots\lfloor\alpha-\frac12\rfloor$, generates a refinable function
$\varphi$ via the cascade algorithm \eqref{eq Cascade Algorithm}.
\end{thm}

\begin{proof} We only need to  check that $H_0$ satisfies the conditions in Lemma \ref{L L2 Konvergenz infinite product}.
Obviously, $H_{0}$ is defined for all $\gamma \in \R$ and all $z \in \C_{\geq 1}$. $H_{0}$ is continuous in both those variables. Moreover,
\begin{eqnarray}
H_{0}(0) = (\cos^{2}(0))^{z} \sum\limits_{k=0}^{\ell} \binom{z+\ell}{k} (\cos^{2}(0))^{\ell-k} \sin^{2}(0)^{k} =1.
\nonumber
\end{eqnarray}
In order to validate the second  criterion, we use again the substitution $x = \sin^{2} \pi \gamma \in [0,1]$ and consider first the case $\ell = 0$. A simple computation gives
\begin{align*}
H_0 (x) & = (1-x)^z = 1 - z x + x^2 \sum_{k=0}^\infty (-1)^k \binom{z}{k+2} x^k =: 1 - zx + x^2 S(x),
\end{align*}
where $S$ is a bounded function on $[0,1]$. Hence, there exist a constant $C > 0$ such that for all $x \in [0,1]$ the following estimate holds:
\[
|H_{0}(x)-1| \leq C |x|.
\]
For general $\ell\in \N$ we express the finite sum in $H_0 (x)$ in the following form:
\begin{align*}
\sum\limits_{k=0}^{\ell} \binom{z+\ell}{k} (1-x)^{\ell-k} x^{k} & = 1+ (z+\ell) (1-x)^{\ell-1} x + \sum\limits_{k=2}^{\ell} \binom{z+\ell}{k} (1-x)^{\ell-k} x^{k}\\
& =: 1 + (z+\ell) x + x^2 P(x),
\end{align*}
where $P$ is an algebraic polynomial of degree at most $\ell$ on $[0,1]$ and thus bounded. Hence,
\begin{align*}
H_{0}(x) & =  (1-x)^{z} \sum\limits_{k=0}^{\ell} \binom{z+\ell}{k} (1-x)^{\ell-k} x^{k}\\
& = \left[1 - zx + x^2 S(x)\right] \left[1 + (z+\ell) x + x^2 P(x)\right]\\
& =: 1+ \ell x + x^{2} J(x),
\end{align*}
where $J$ is a bounded function on $[0,1]$. Hence, for all $x \in [0,1]$ there exists a constant $C>0$ such that
\[
|H_{0}(x)-1| \leq C |x|.
\]
Thus, criterion $2)$ is fulfilled for the filter $H_{0}$ and as a consequence of Lemma \ref{L L2 Konvergenz infinite product}, the infinite product \eqref{eq infinite product} converges pointwise and uniformly on compact sets. Finally, condition 3) is satisfied by Proposition \ref{9201}.
\end{proof}

\begin{rem}
Comparing the formal limit $\widehat\varphi$ of the cascade algorithm with the formal infinite product \eqref{60131b}, we make the choice $\widehat{\varphi}(0) := 1$.
\end{rem}

Following the terminology in \cite{DRoSh3}, we call  the function $\varphi$ constructed in Theorem \ref{60131a}  a {\em pseudo-spline of complex order $(z, \ell)$} or for short a \emph{complex pseudo-spline}.

\ml
\begin{rem}
If $z:=m\in \N$ in \eqref{eq1} our definition corresponds to the definition of pseudo-spline of type II as given in \cite{DRoSh3}. In \cite[Remark 3]{Zhuang} it was claimed that the concept of pseudo-spline can be extended to fractional powers by considering $\cos^r\pi\gamma$ instead of $\cos^{2m}\pi\gamma$, where $r\in \R$ with $r\geq 2m\in \N$. This, however, is not true since for negative values of the $\cos$-function, the expression $\cos^r\pi\gamma = (\cos\pi\gamma)^r$ is no longer well-defined. For this reason, it is essential to consider $H_0$ as a function in $\cos^2$ and $\sin^2$.
\end{rem}

\begin{figure}[h!]
\centering
\includegraphics[height = 4cm, width = 6cm]{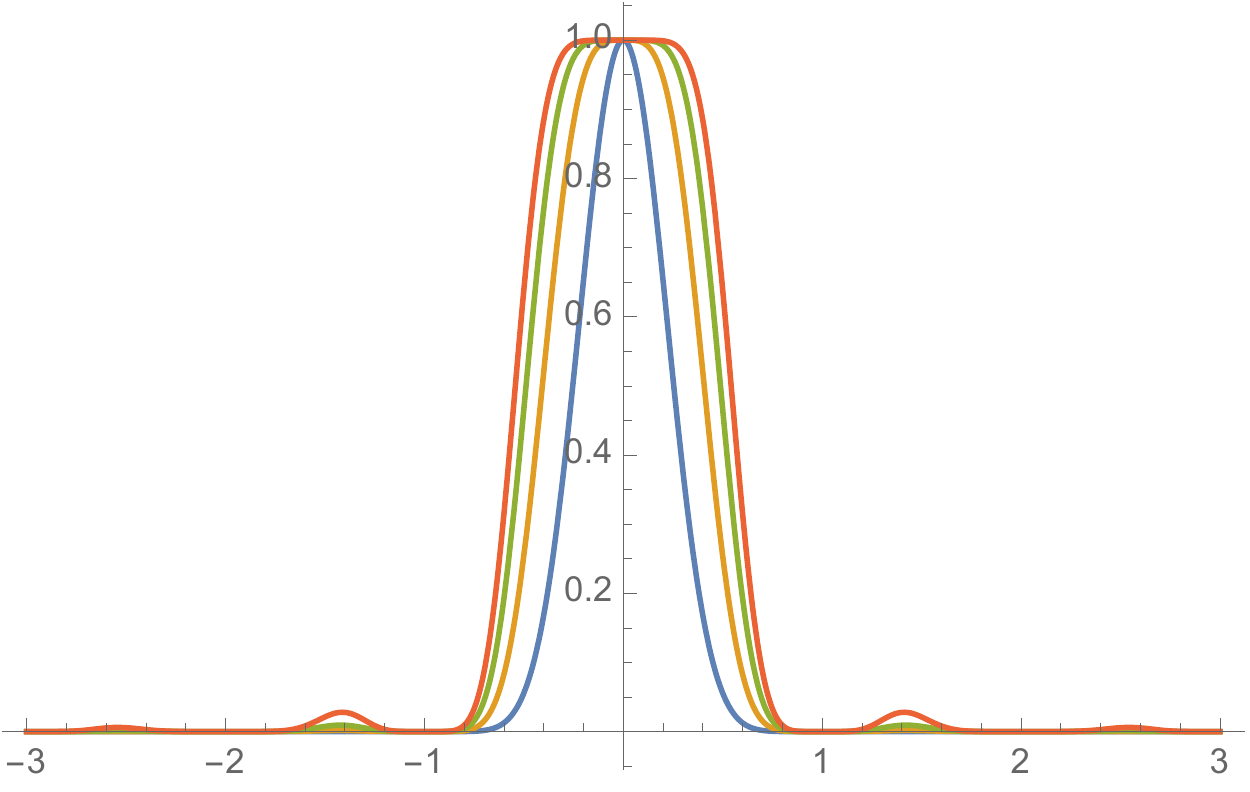}\\[2ex]
\includegraphics[height = 4cm, width = 6cm]{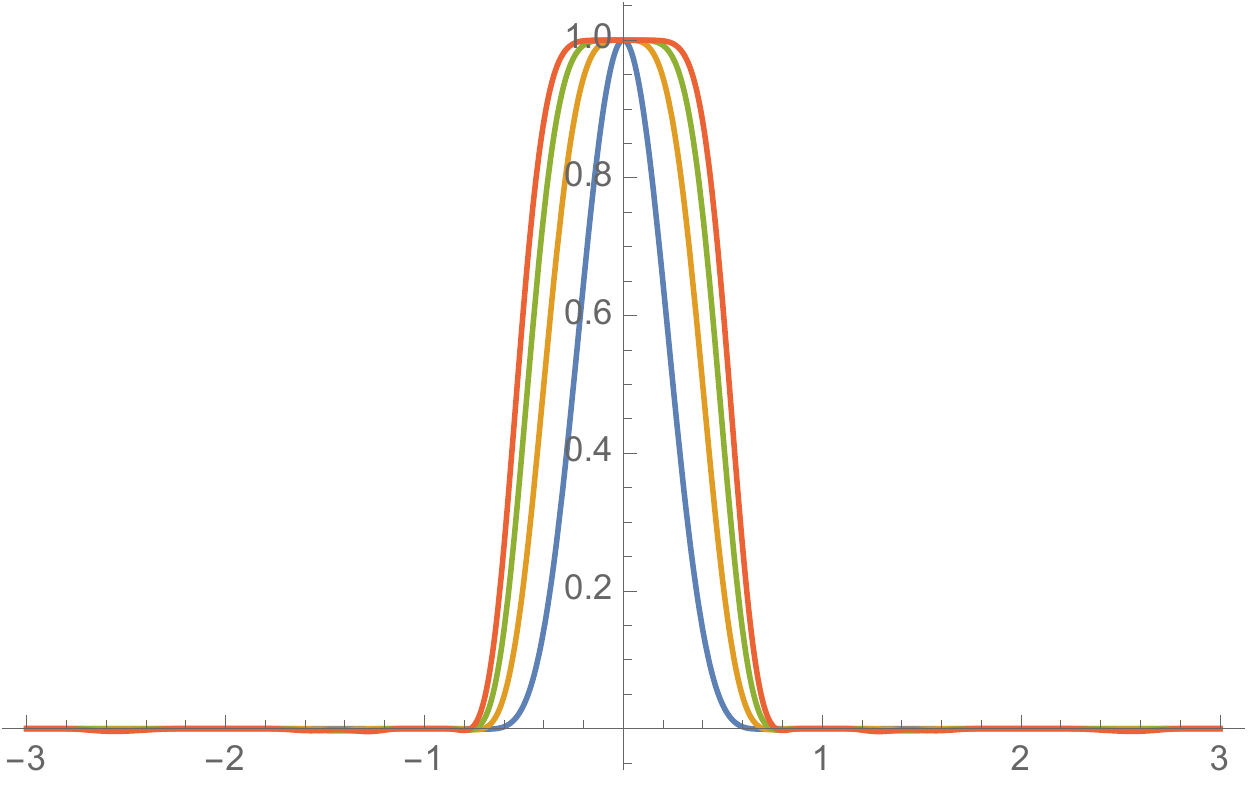}
\includegraphics[height = 4cm, width = 6cm]{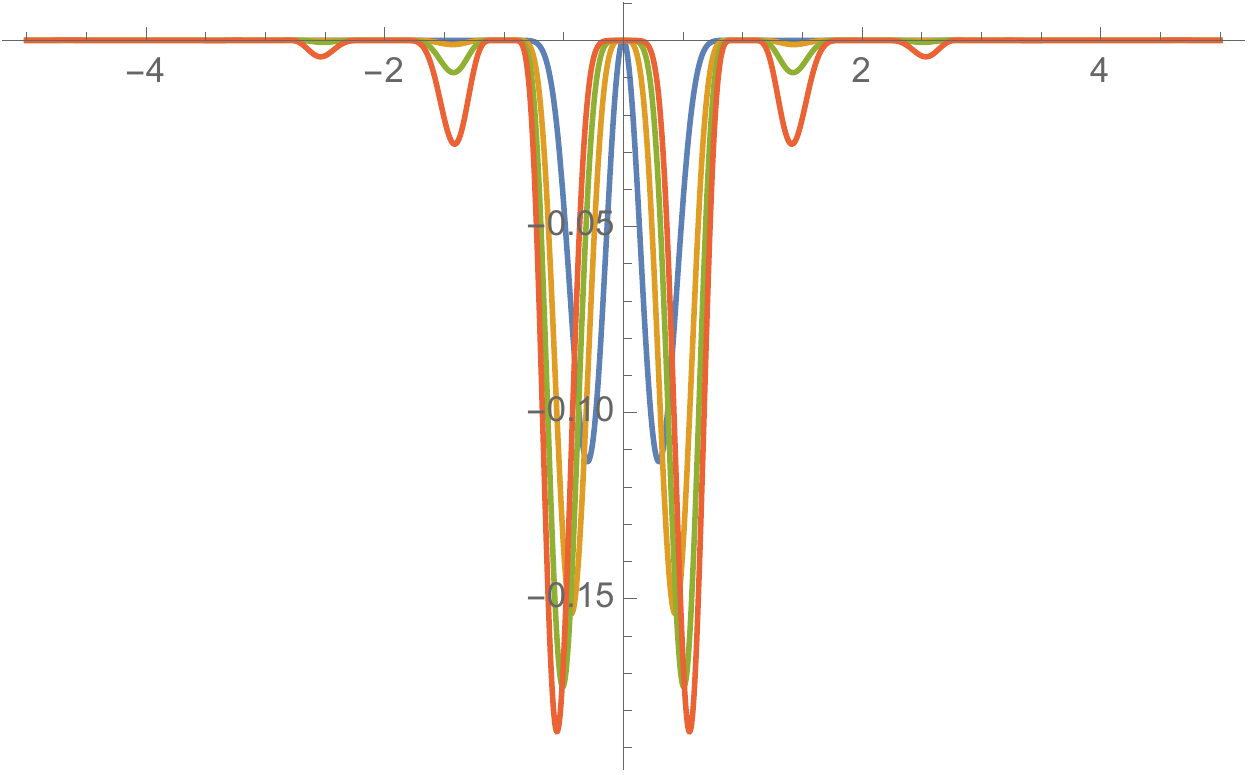}
\caption{The pseudosplines in Fourier domain  $\widehat{\varphi}^{(3.2 + i,\ell)}$ for $\ell= 0, 1,2,3$ (inner to outer curves). The upper image shows $|\widehat{\varphi}^{(3.2 + i,\ell)}|$, the lower left image $\re \widehat{\varphi}^{(3.2 + i,\ell)}$, and the lower right image $\im \widehat{\varphi}^{(3.2 + i,\ell)}$. The pseudospline parameter $\ell$ allows for a tuning of the width of the lowpass property of the refinable function.}
\end{figure}

\begin{figure}[h!]
\centering
\includegraphics[width = 0.3\textwidth]{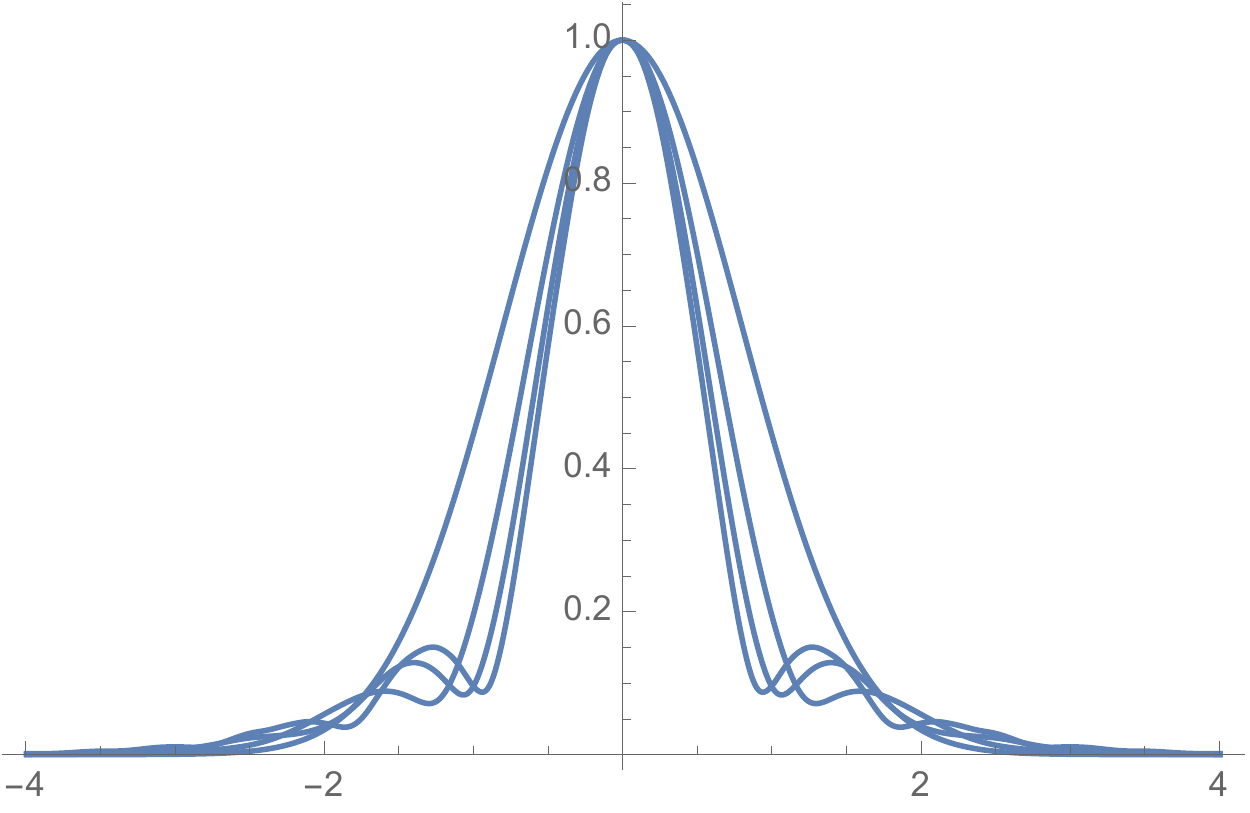}

(a)

\bigskip

\includegraphics[width = 0.3\textwidth]{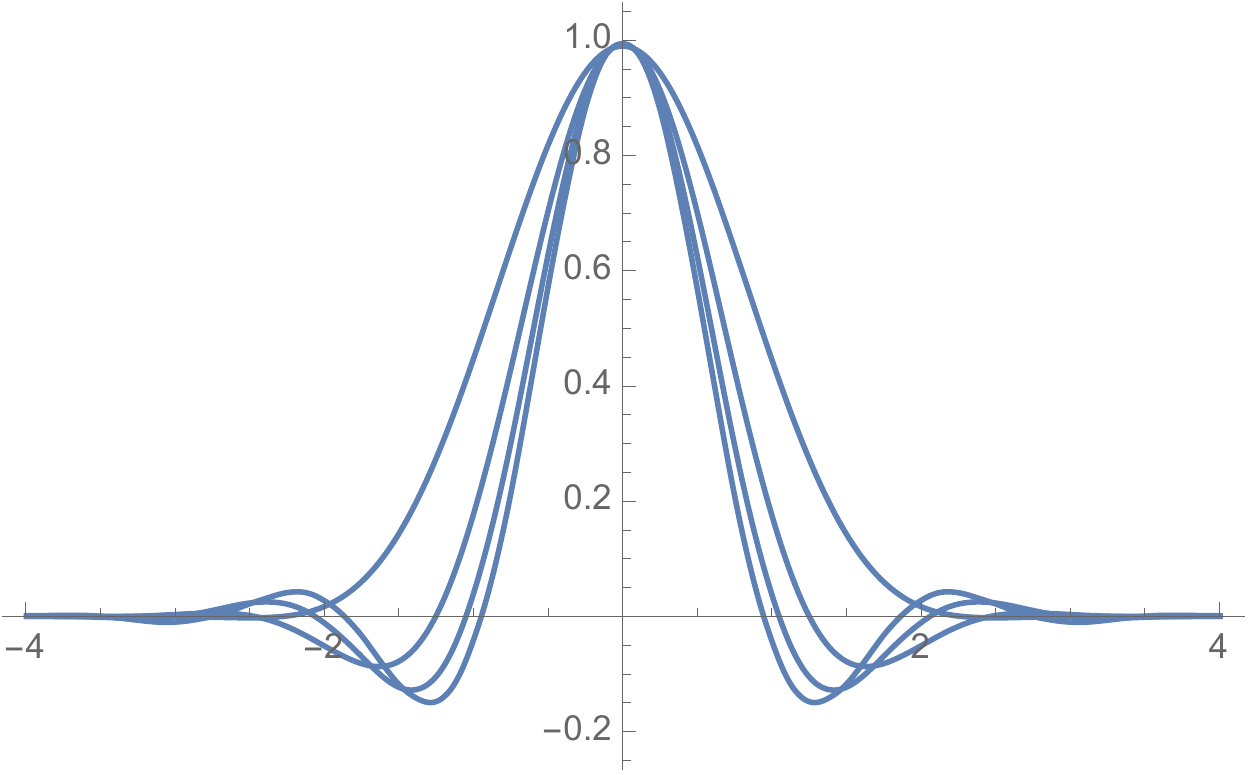}\hspace{2cm}
\includegraphics[width = 0.3\textwidth]{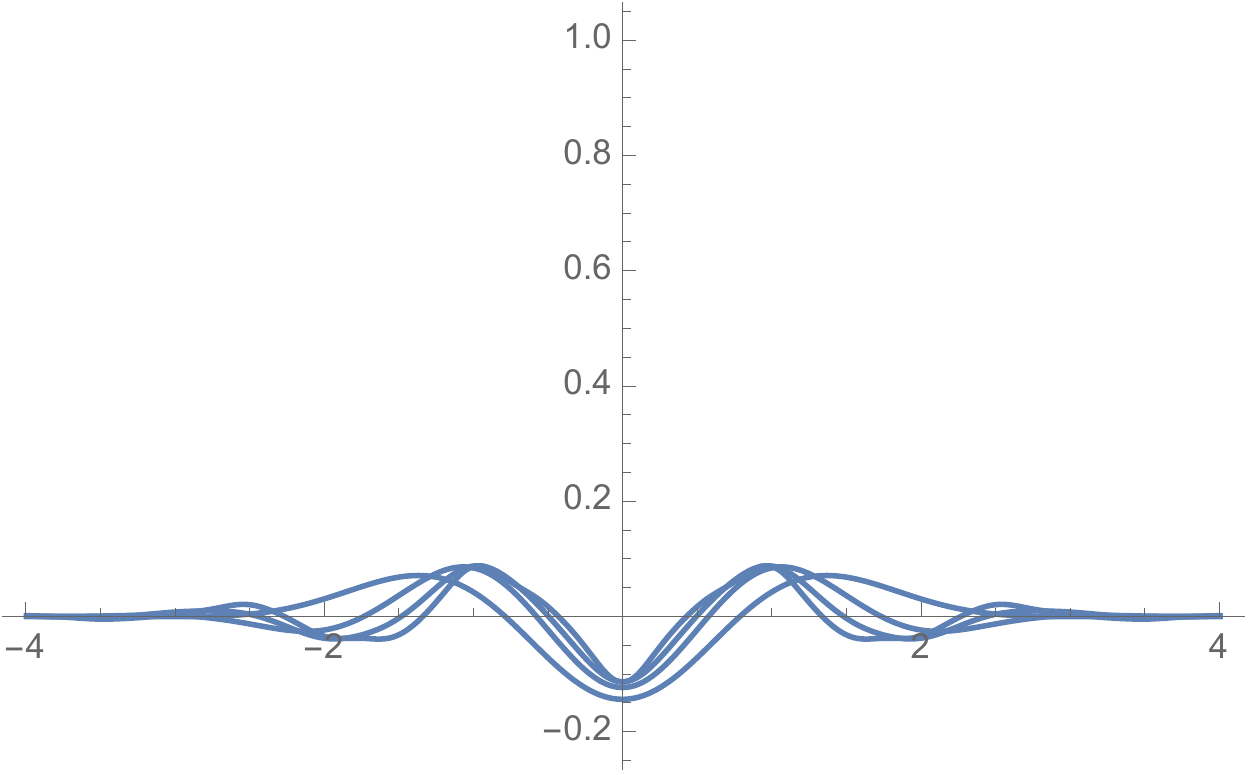}

(b) \hspace{145pt} (c)

\label{fig Pseudosplines zeitbereich}
\caption{Pseudo-splines ((a) modulus, (b) real and (c) imaginary part) in time domain. Parameters $z = 3.2 + i$, $\ell = 0,1,2,3$.}
\end{figure}
%
%
%
%
%
%
%
%
%
%

As a variant to the construction above, one can also consider complex pseudo-splines with a fixed shift $w\in \C$ (cf. \cite{forster06}), namely,
\begin{align*}
&H_{0}^{(z,\ell,w)}(\gamma)\\
& := \frac{1}{2^{2z}}(1+e^{2\pi i \gamma})^{z-w} (1+e^{-2\pi i\gamma})^{z+w} \sum\limits_{k=0}^\ell \binom{z + \ell}{k} (\sin^2\pi\gamma)^k\,(\cos^2\pi\gamma)^{\ell-k},
\end{align*}
for $z\in\C_{\geq 1}$ and $0 \leq \ell \leq \lfloor \re z -\frac12 \rfloor$, in order to define fractional and complex pseudo-splines. In the case where $\ell:=0$ it was shown in \cite{luisier} that the shift allows for better adaption to the signal or image.

A straightforward computation shows that
\be
H_{0}^{(z,\ell,w)}(\gamma) = e^{-2\pi i w \gamma}\,H_0^{z,\ell} (\gamma), \quad\gamma\in \T.
\ee
Setting $w := u + i v$, $u,v\in \R$, one quickly verifies that \eqref{eq4} holds true only if $v =0$. Hence $H_{0}^{(z,\ell,w)}$ is  related to $H_{0}^{(z,\ell)}$ via a pure phase factor:
\[
H_{0}^{(z,\ell,u)}(\gamma) = e^{-2\pi i u \gamma}\,H_0^{z,\ell} (\gamma), \quad\gamma\in \T.
\]
Clearly, $H_{0}^{(z,\ell,u)}(\gamma)$ is $1$-periodic function and satisfies $H_{0}^{(z,\ell,u)}(0) = 1$. Using the substitution $x = \sin^2\pi\gamma$, expressing the phase factor $e^{-2\pi i u \gamma}$ as a function of $x$, and writing out its Taylor expansion around $x=0$, one obtains
\[
e^{-2\pi i u \gamma} = e^{-2 i u \arcsin\sqrt{x}} = 1 - 2 i u \sqrt{x} - 2 u^2 x + \sqrt{x}\, F(x),
\]
where $F$ is a bounded function on $[0,1]$. As $H_0^{(z,\ell)}$ can be written in the form $H_0^{(z,\ell)} (x) = 1 + c \,x + x^2 M(x)$ for some $c\in \C$ and a bounded function $M$ on [0,1] (see above), we obtain for $H_{0}^{(z,\ell,u)}(x)$ the expression
\begin{align*}
H_{0}^{(z,\ell,u)}(x) & = \left[1 - 2 i u \sqrt{x} - 2 u^2 x + \sqrt{x} F(x)\right] \left[1 + c \,x + x^2 M(x)\right]\\
& = 1 + (F(x) - 2i u)\sqrt{x} + x\, G(x),
\end{align*}
for some function $G$ bounded on $[0,1]$. Hence, there exists a positive constant $C$ so that
\[
|H_{0}^{(z,\ell,u)}(x) - 1| \leq C |x|^\frac12.
\]
Consequently, Proposition \ref{9201} also applies to the shifted filters $H_{0}^{(z,\ell,u)}$ and by Theorem \ref{60131a} the shifted filters $H_{0}^{(z,\ell,u)}$ generate refinable functions via the cascade algorithm.
\section{Lowpass Properties of Complex Pseudo-Splines}
We first consider the properties of the complex pseudo-splines in the Fourier domain.
Because of the lower inequality in condition 3) of Lemma  \ref{L L2 Konvergenz infinite product}, we have  that $0<\vartheta\leq |H_{0}(\gamma)|^{2} \leq 1$, for all $\gamma \in\R$. As a consequence, the pseudo-spline $\varphi$ is bounded above by one in the Fourier domain, i.e., $|\widehat{\varphi}(\gamma)| \leq 1$, for all $\gamma \in \R$.

As expected for a refinable function, the pseudo-splines act as lowpass filters. In fact, there exists a neighborhood of the origin, where $\widehat{\varphi}$ does not vanish. To see this, we need some preparations.
\begin{lemma}\label{60129c}
Let $n\in\N$ and let $\{z_j := x_j + y_j : j = 0,1,\ldots, n\}$ be a set of complex numbers. Then
\[
\Arg \prod_{j=0}^n z_j = \sum\limits_{j=0}^n \tan^{-1}\frac{y_j}{x_j},
\]
where $\Arg$ denotes the principal argument of a complex number and has values in $(-\pi,\pi]$.  If, in addition, each $z_j$ is of the form $z_j := x+j + i y$, for given $(x,y)\in\R$ with $x > 1$, then
\[
\tan^{-1} \frac{y}{x+j} \in \left(-\frac\pi2,\frac\pi2\right),\quad\forall\,j=0,1,\ldots,n,
\]
and
\be\label{eq9}
\left\vert\tan^{-1} \frac{y}{x+j}\right\vert \leq \left\vert\tan^{-1} \frac{y}{x+j-1}\right\vert,\quad\forall\,j=1,\ldots,n.
\ee
\end{lemma}
\begin{proof}
Write each $z_j$ in the form $z_j = |z_j|\,e^{i\theta_j}$ with $\theta_j := \tan^{-1}\frac{y_j}{x_j}$.
\end{proof}

Set $H_{0}^{(z,\ell)}(\gamma) := H_{0}^{(z,0)}(\gamma) P^{(z,\ell)}(\gamma)$, where
\be\label{Pzl}
P^{(z,\ell)}(\gamma) := \sum\limits_{k=0}^\ell \binom{z + \ell}{k} (\sin^2\pi\gamma)^k\,(\cos^2\pi\gamma)^{\ell-k}.
\ee

\bpr
Suppose $z := x + i y$, $x\geq 1$, is such that
\be\label{eq10a}
\sum\limits_{j=0}^\ell \tan^{-1}\frac{y}{x+j} \in \left(-\frac\pi2,\frac\pi2\right)
\ee
holds. Then there exists a positive constant $c>0$ that bounds $\widehat{\varphi}$ from below in a neighborhood of the origin, i.e.,
\begin{equation}
0 < c \leq |\widehat{\varphi}^{(z,0)}(\gamma)| \leq |\widehat{\varphi}^{(z,\ell)}(\gamma)|.
\label{eq fractionelle splines sind positiv}
\end{equation}
\epr

\begin{proof} The result in  \eqref{eq fractionelle splines sind positiv} clearly holds for all those tuples $(z,\ell)$, for which
$$
0 < \widetilde{c} \leq |H_{0}^{(z,0)}(\gamma)| \leq   |H_{0}^{(z,0)}(\gamma) P^{(z,\ell)}(\gamma)|  =|H_{0}^{(z,\ell)}(\gamma)|
$$
holds for $\gamma$ in a neighborhood of the origin. Here, the third inequality holds whenever $|P^{(z,\ell)}(\gamma)| \geq 1.$
For fractional $z = \alpha \geq 1$, this is obviously true. In fact, consider the representation as an algebraic polynomial. Then
$$
P^{(\alpha, \ell)}(x) = 1 + \sum\limits_{k=1}^{\ell}\binom{\alpha-1+k}{k} x^{k} \geq 1 \quad \mbox{for } x \in [0,1],
$$
because of the positivity of the sum. For complex $z$ we need more sophisticated estimates. Consider
\begin{eqnarray*}
|P^{(z, \ell)}(x)|^{2} & = & \left( \sum\limits_{k=0}^{\ell}\re \binom{z-1+k}{k}x^{k}\right)^{2} +\left( \sum\limits_{k=0}^{\ell}\im \binom{z-1+k}{k}x^{k}\right)^{2}
\nonumber\\
& = & \left(1 + \sum\limits_{k=1}^{\ell}\re \binom{z-1+k}{k}x^{k}\right)^{2} +\left( \sum\limits_{k=1}^{\ell}\im \binom{z-1+k}{k}x^{k}\right)^{2}
\\
& \geq & 1 + 2\sum\limits_{k=1}^{\ell}\re \binom{z-1+k}{k}x^{k} + \left(\sum\limits_{k=1}^{\ell}\re \binom{z-1+k}{k}x^{k}\right)^{2}.
\end{eqnarray*}
Lemma \ref{60129c} allows us to estimate the binomial terms.

Consider a fixed $z := x +i y$ with $x\geq 1$. Then
\[
\re\binom{z-1+k}{k} = \frac{\re\prod\limits_{j=0}^{k-1} (z+j)}{k!}.
\]
By \eqref{eq9}, it suffices to require that
\[
\sum\limits_{j=0}^\ell \tan^{-1}\frac{y}{x+j} \in \left(-\frac\pi2,\frac\pi2\right)
\]
to guarantee that $\re\binom{z-1+k}{k} \geq 0$ for all $k=1, \ldots, \ell$.

Thus, for all $z\in \C_{\geq 1}$ satisfying \eqref{eq10a} the polynomial $|P^{(z,\ell)}(\gamma)| \geq 1$ and \eqref{eq fractionelle splines sind positiv} holds.
\end{proof}

Note that for $\ell = 0$, condition \eqref{eq10a} is satisfied as expected since $H^{(z,0)}$ is the filter of the classical fractional or complex B-spline.

\begin{rem} The refinable functions generated by the shifted filters $H_{0}^{(z,\ell,u)}$ also satisfy Proposition \ref{60129c} since they differ from the filters $H_{0}^{(z,\ell)}$ only by the phase factor $e^{-2\pi i u \gamma}$.
\end{rem}
\section{Construction of Parseval Wavelet Frames} \label{2803a}
In this section we will use the unitary extension principle (UEP) by Ron and Shen \cite{RoSh2}
to construct Parseval frames, generated by fractional or complex pseudo-splines. This
generalizes some of the results in \cite{DRoSh3}, where frame constructions based on
the ``integer pseudo-splines''  were presented.  Since we only want to illustrate
the use of pseudo-splines of complex order we do not aim at the most general version
of the result but confine ourself to a construction with three generators.

\bpr \label{9603}
Consider the filter $H_0$ in  \eqref{eq1} and the associated refinable function
$\varphi$ in \eqref{eq infinite product}. Furthermore, let
\bee \label{9403}
\eta(\ga):= 1- \left(|H_0(\ga)|^2+ |H_0(\ga + \tfrac12)|^2\right).
\ene
Let $\sigma$ be a $1$-periodic function such that $|\sigma(\ga)|^2= \eta(\ga),$ and define the filters $\{H_n\}_{n=1}^3$ by \bes \label{9303}
H_1(\ga)= e^{2\pi i \ga}\overline{H_0(\ga+ \tfrac12)}, \ \
H_2(\ga)=\frac1{\sqrt{2}}\sigma(\ga), \, \,  H_3(\ga)=\frac1{\sqrt{2}}\, e^{2\pi i \ga}\sigma(\ga).
\ens
Then the functions $\{\psi_n\}_{n=1}^3$ given by
\bee \label{2002}
\widehat{\psi_n}(2\ga)=H_n(\ga)\widehat{\varphi}(\ga)
\ene
generate a Parseval frame $\{D^jT_k\psi_{n}\}_{j,k\in \mz, n=1,
2,3}$ for $\ltr$. \epr

\bp We have already seen that
$$
\widehat{\varphi} (\gamma) = \prod\limits_{m>0} H_{0}(2^{-m}\gamma) = H_{0}(\gamma/2) \widehat{\varphi} (\gamma/2).
$$ Note that the formulation of Proposition \ref{9603} corresponds
precisely to the setup on page 21 in \cite{DRoSh3}, except that the filter
$H_0$ in general is not a trigonometric polynomial;  it is easy to see
that the arguments in \cite{DRoSh3} still go through.
According to the general setup in the unitary extension principle
(see again page 21 in  \cite{DRoSh3} or Cor. 18.5.2 in \cite{C}  with $\Theta=1$) we only need to check
 that \bei \item[(i)]  $\lim_{\ga\to 0}\widehat{\varphi}(\ga)=1.$
 \item[(ii)] The function $\eta$ in \eqref{9403} is nonnegative.\eni
The condition (i) is clearly satisfied, and (ii) follows from Proposition \ref{9201}.  Thus the
conclusion follows.
\ep

Note that if we  expand the filters $H_n, n=1,2,3$ in  Fourier series, $H_n(\ga)= \sukz c_{k,n}e^{2\pi ik\ga},$
the functions $\psi_n, n=1,2,3,$ are given explicitly by
\[
\psi_n(x)= \sqrt{2}\sukz c_{k,n}DT_{-k}\varphi (x)= 2 \sukz c_{k,n}\varphi (2x+k).
\]
\begin{rem}
Proposition \ref{9603} also holds for the shifted filters $H_0^{(z,\ell,u)}$ and the refinable functions generated by them. Consequently, we also obtain Parseval wavelets frames from these refinable functions.
\end{rem}

\begin{figure}[h!]
\centering
\flushleft{$\ell = 0:$}\\
\includegraphics[width = 0.3\textwidth]{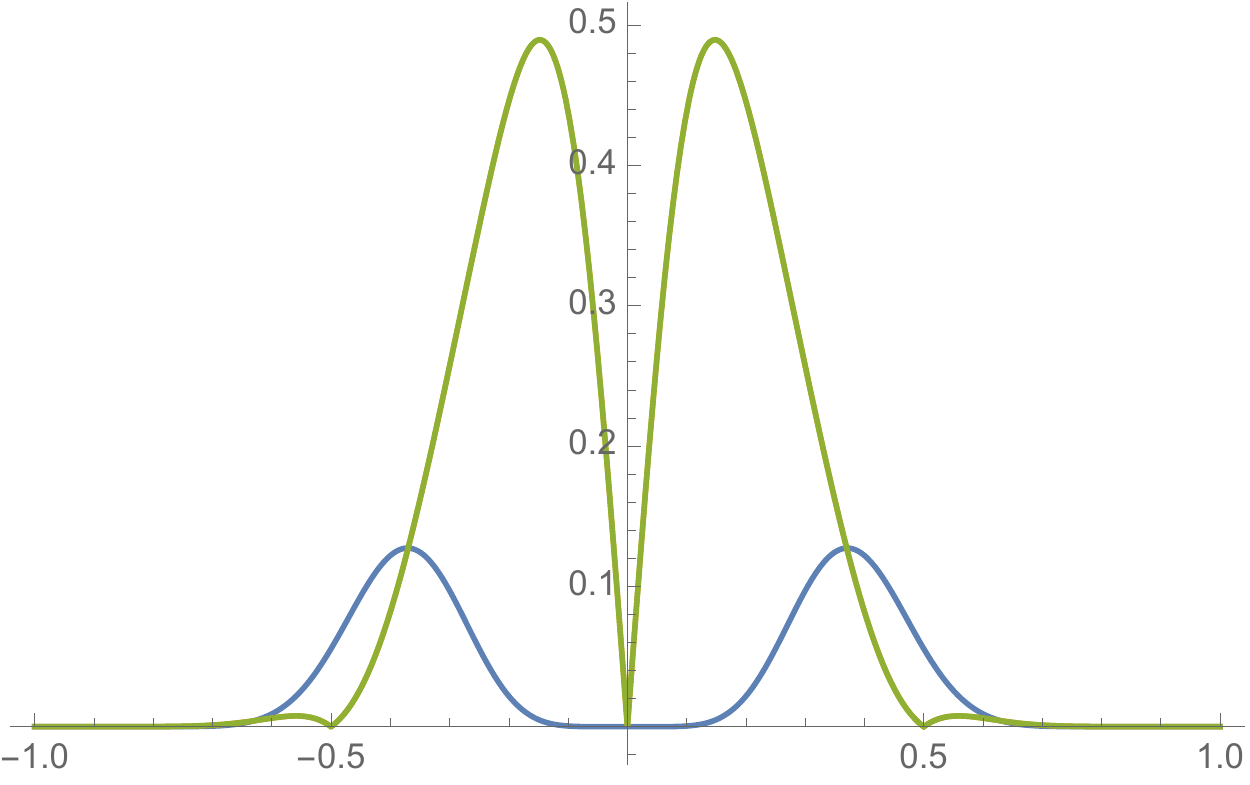}
\includegraphics[width = 0.3\textwidth]{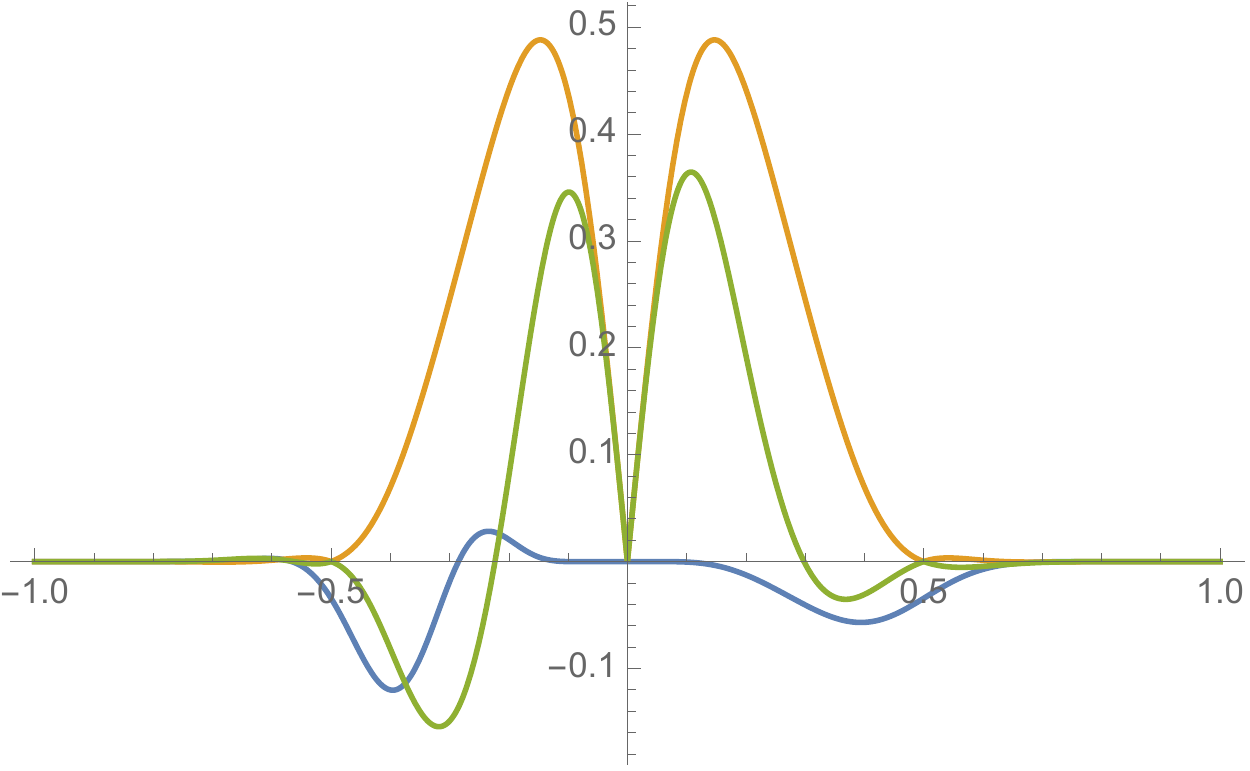}
\includegraphics[width = 0.3\textwidth]{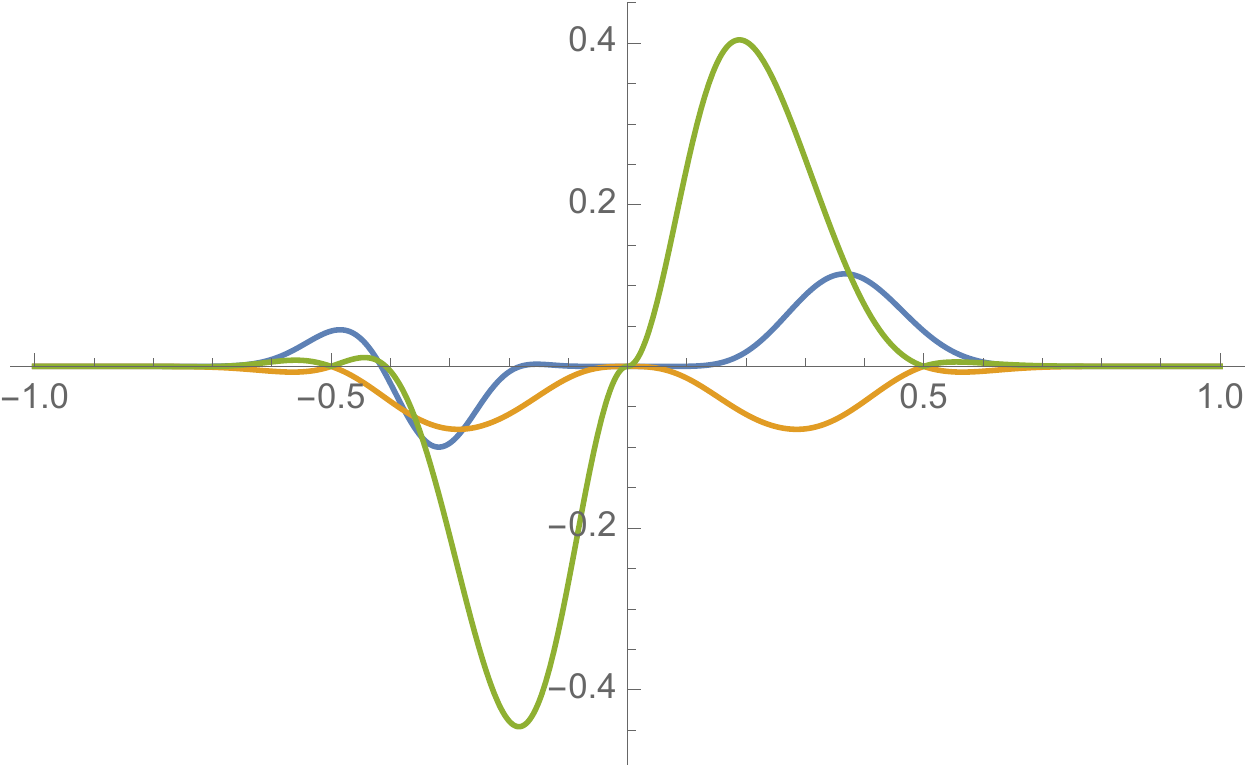}

\flushleft{$\ell = 1:$}\\
\includegraphics[width = 0.3\textwidth]{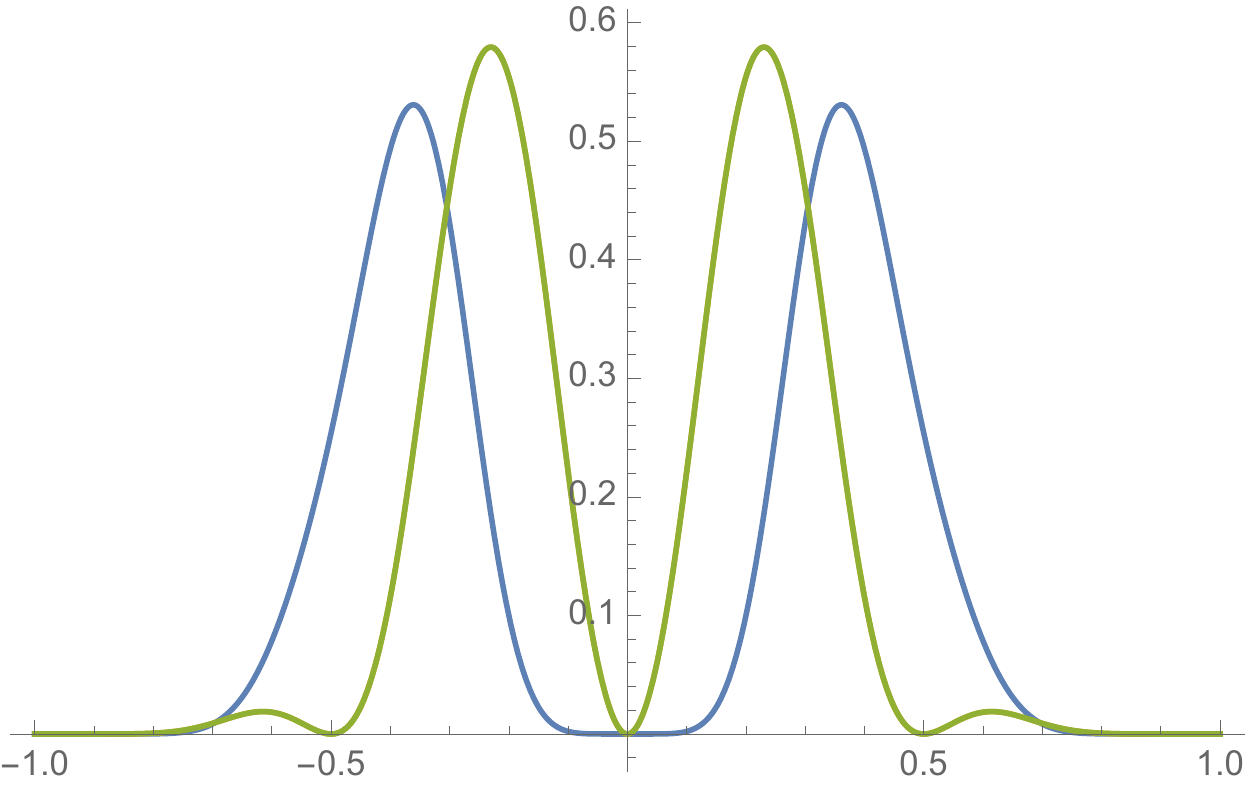}
\includegraphics[width = 0.3\textwidth]{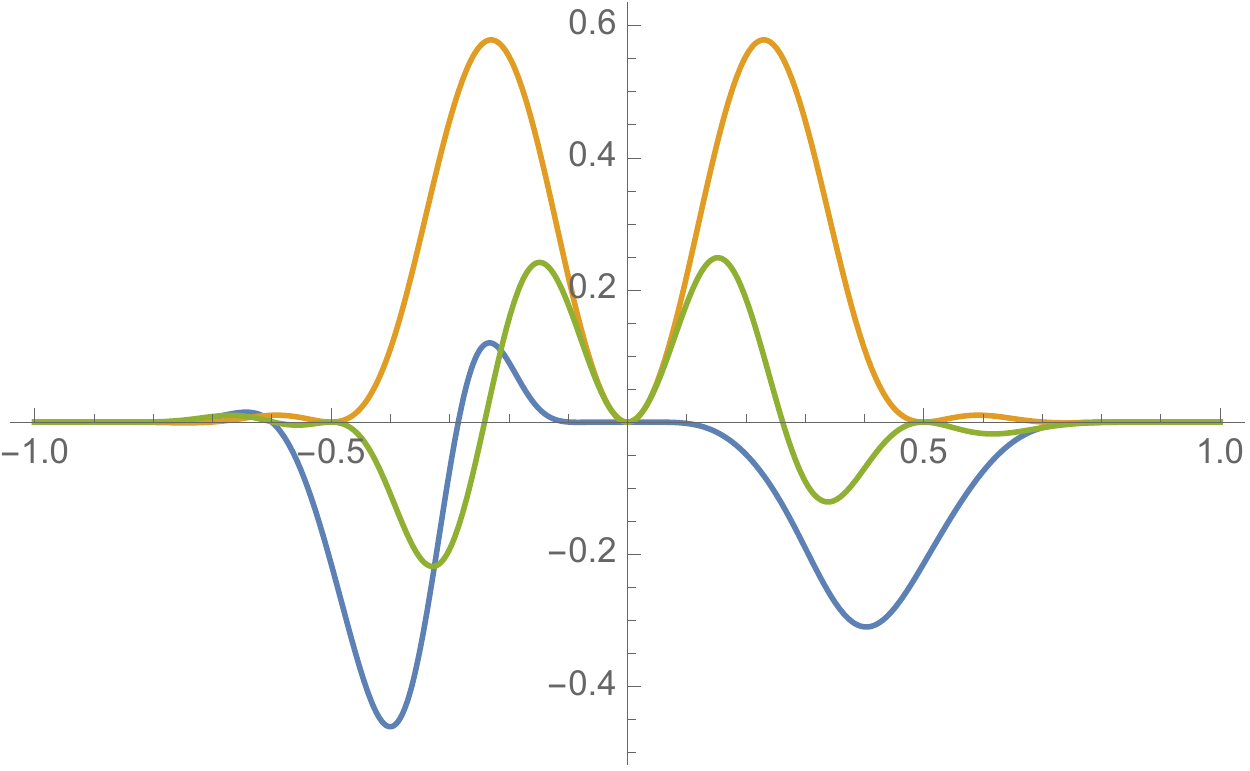}
\includegraphics[width = 0.3\textwidth]{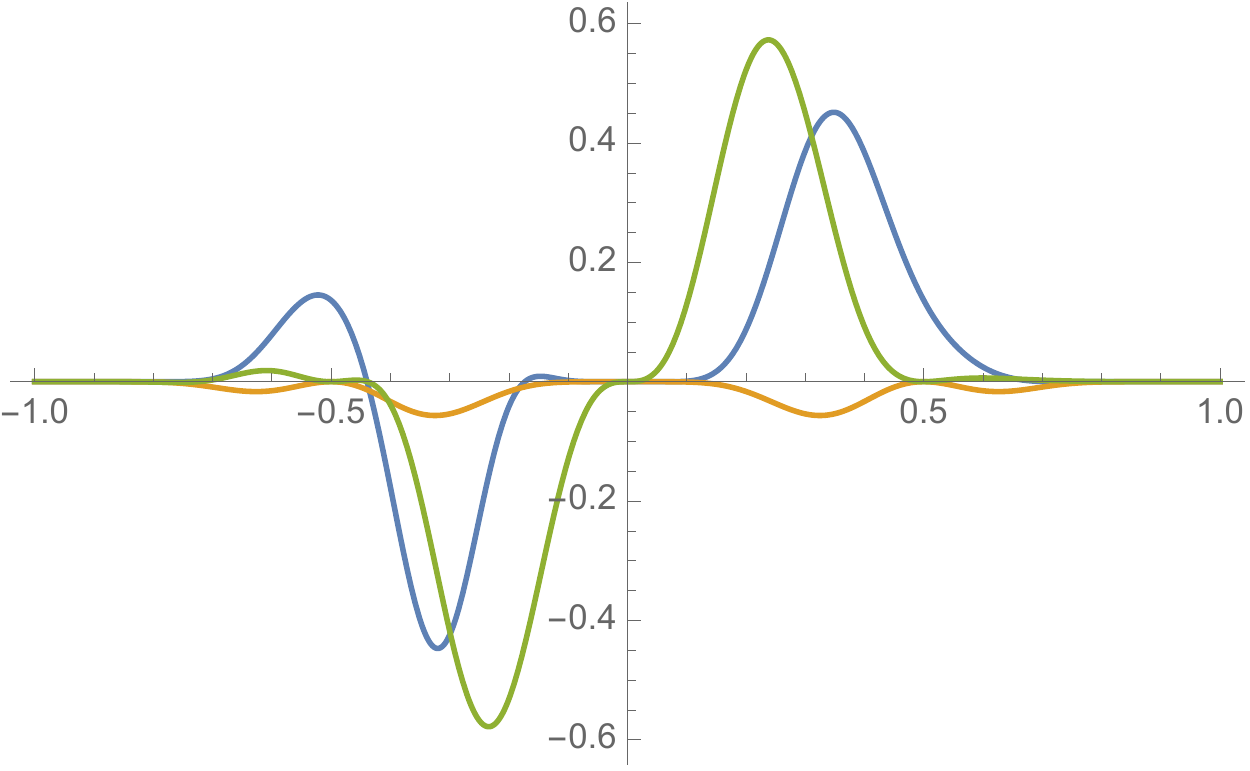}

\flushleft{$\ell = 2:$}\\
\includegraphics[width = 0.3\textwidth]{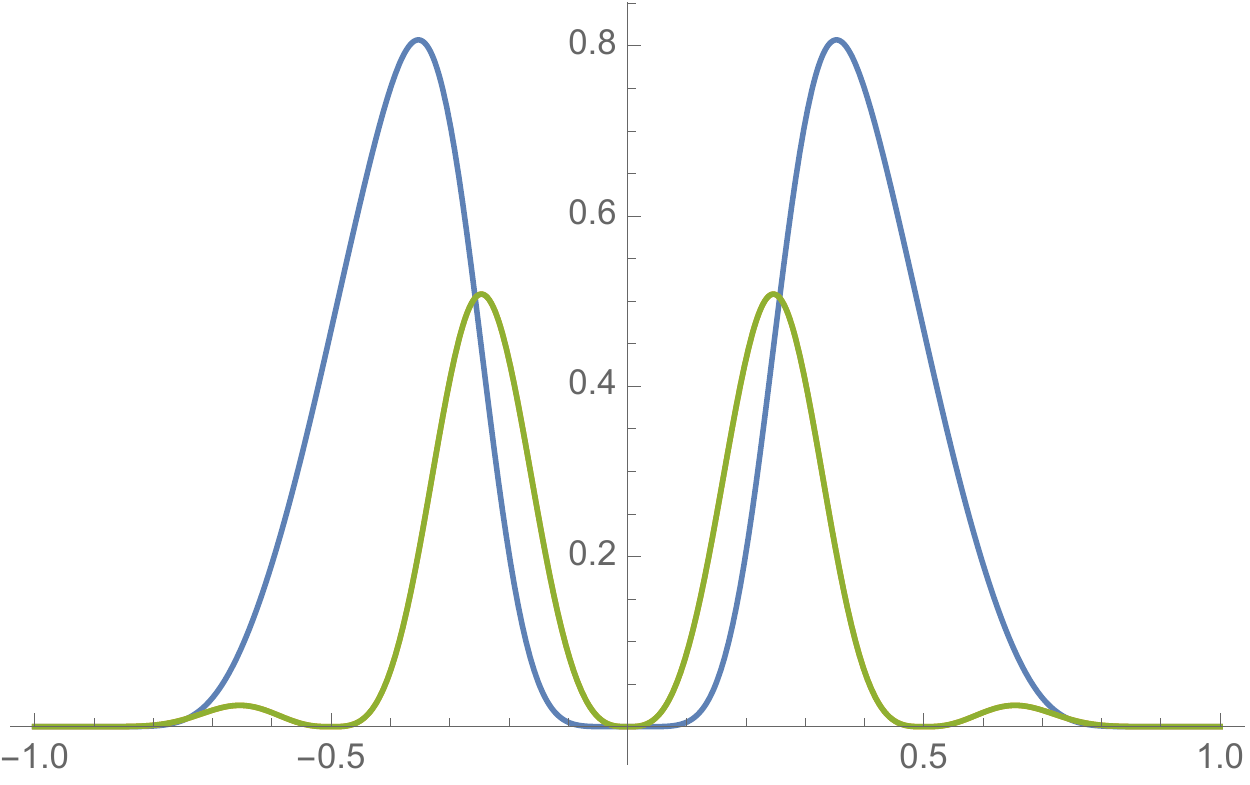}
\includegraphics[width = 0.3\textwidth]{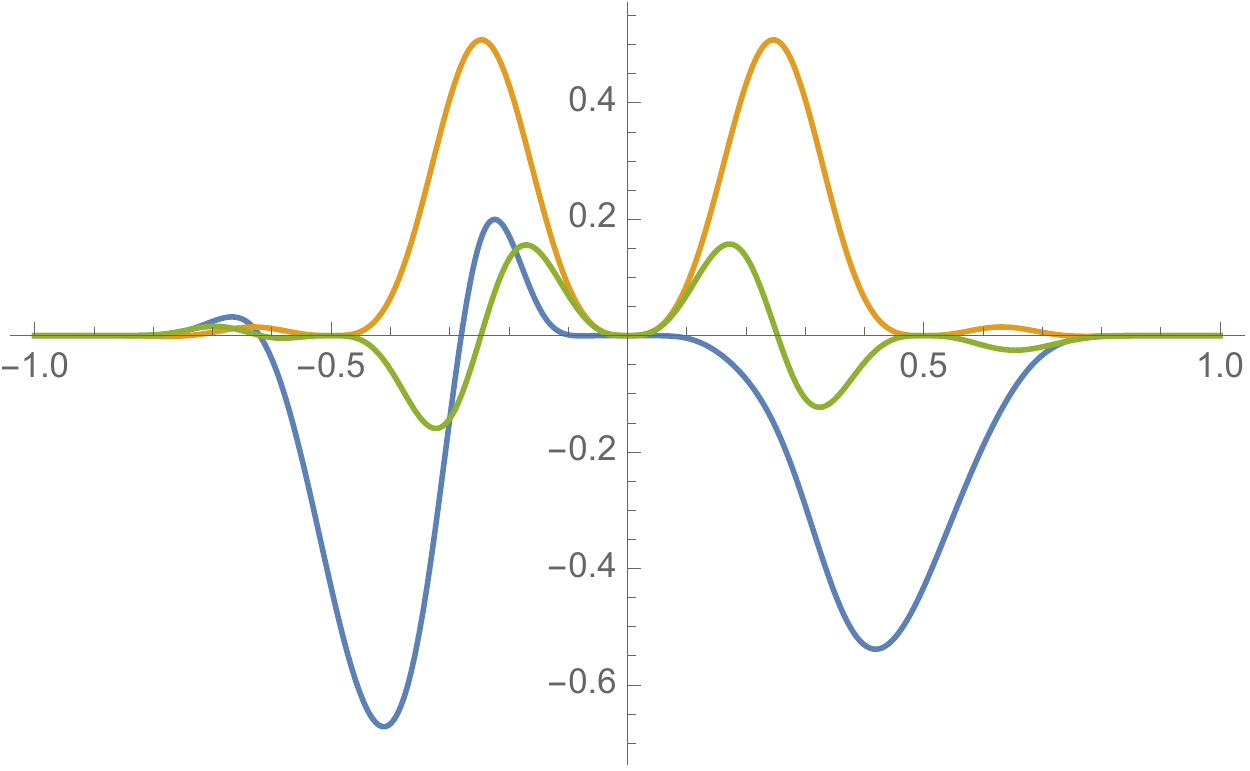}
\includegraphics[width = 0.3\textwidth]{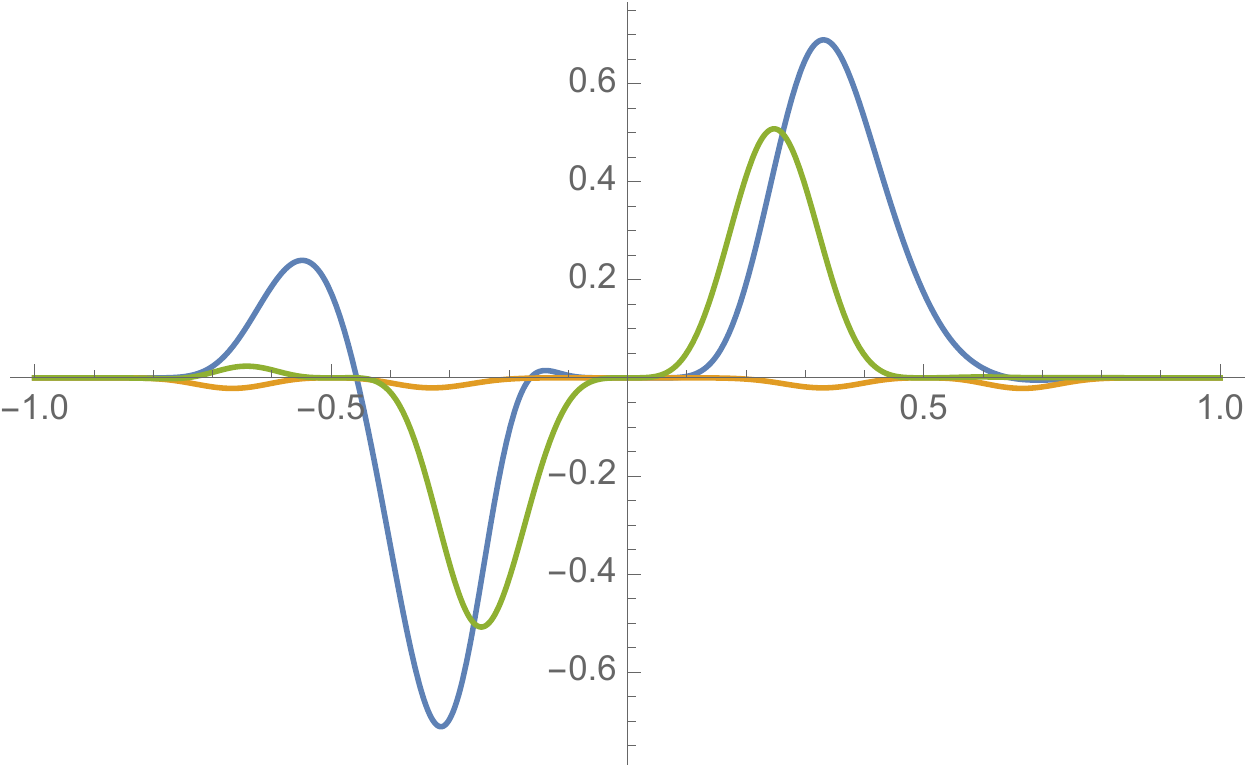}

\flushleft{$\ell = 3:$}\\
\includegraphics[width = 0.3\textwidth]{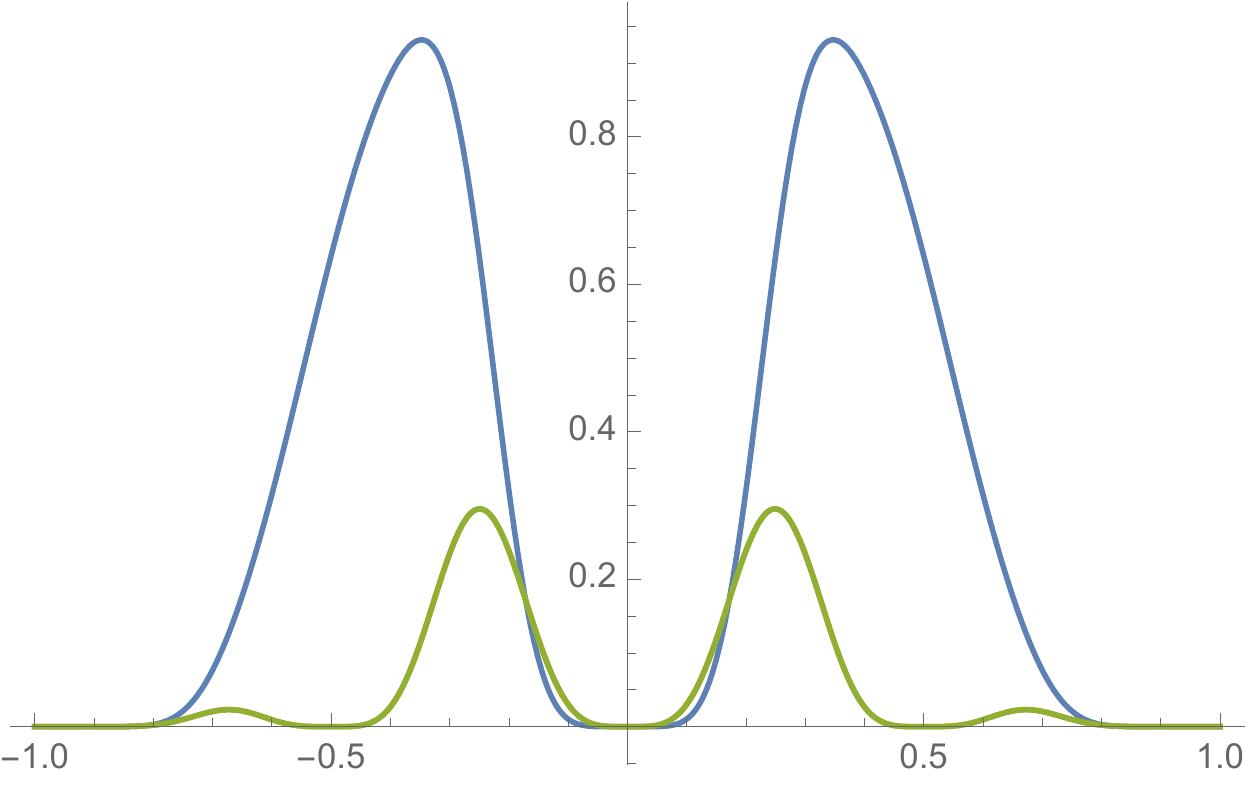}
\includegraphics[width = 0.3\textwidth]{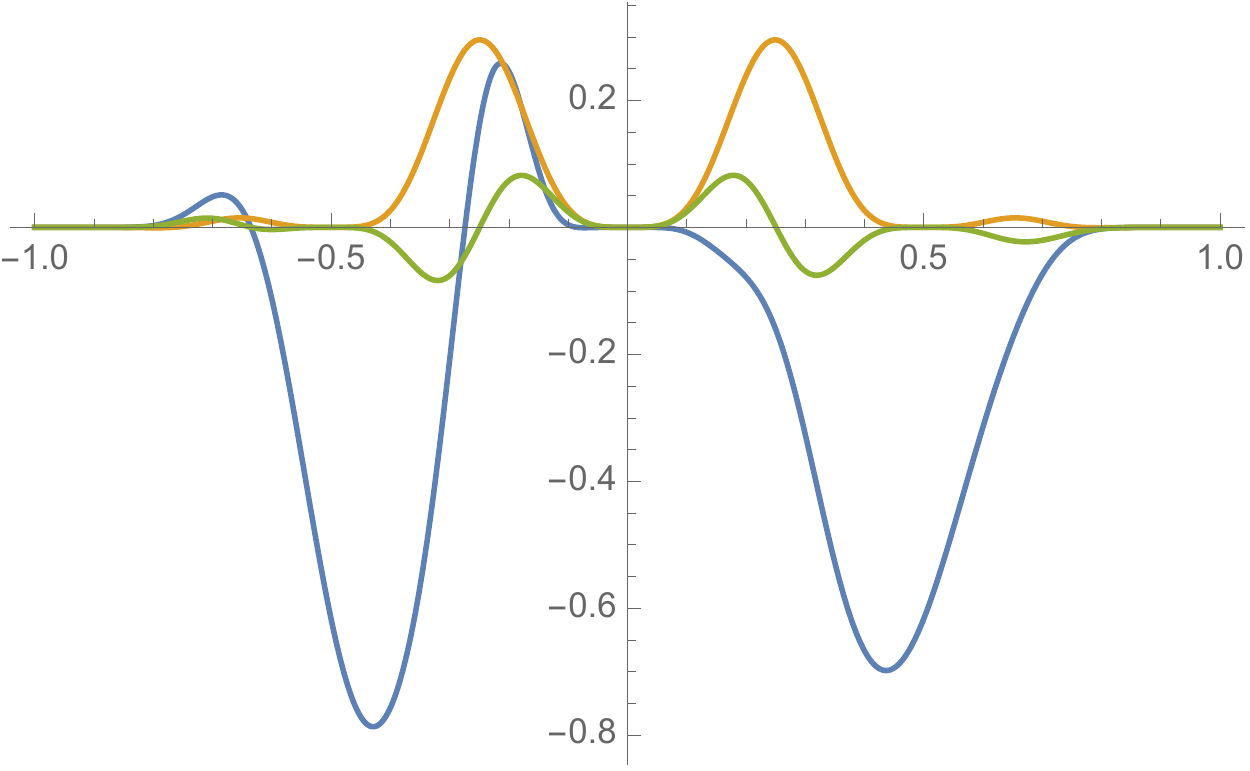}
\includegraphics[width = 0.3\textwidth]{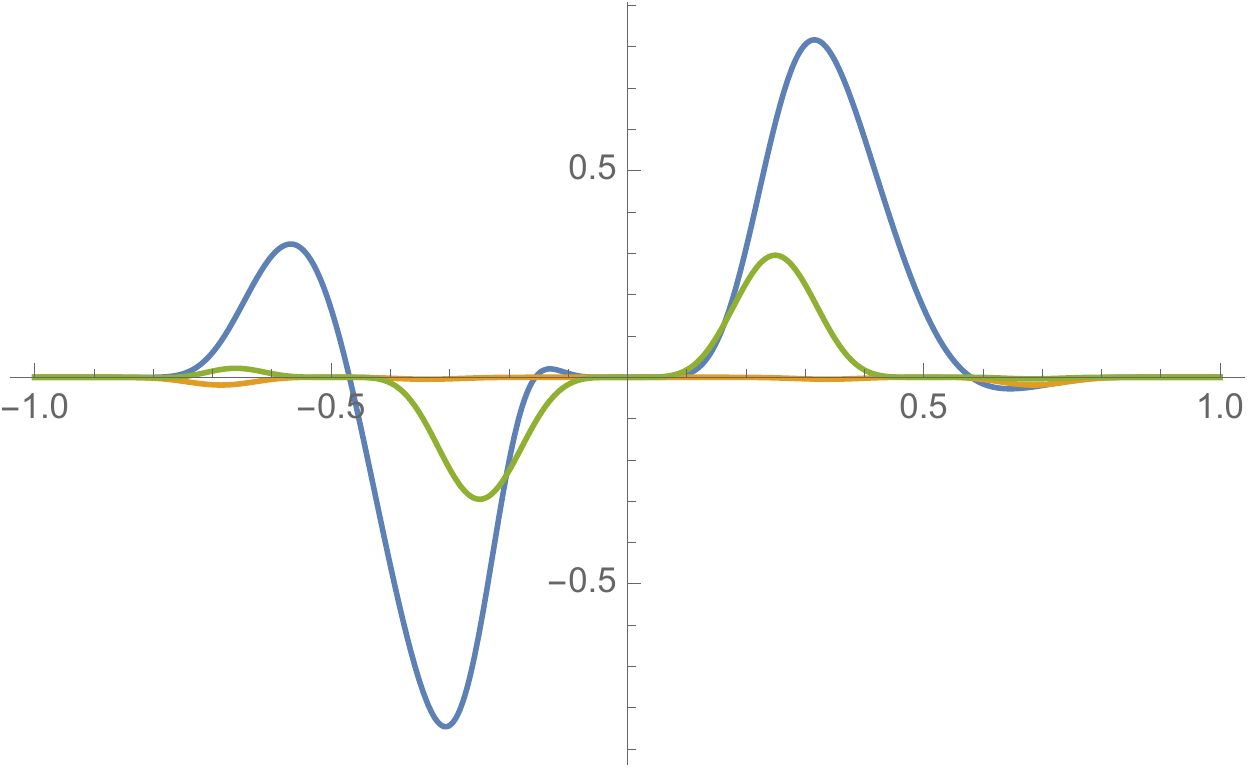}

\label{fig Wavelets im Fourier-Bereich}
\caption{The three framelets $\psi_1$, $\psi_2$, and $\psi_3$ from Proposition \ref{9603} in the frequency domain. Left: modulus, center: real part, right: imaginary part. Note that two of the framelets have the same modulus. Parameter: $z = 3.2 + i$.}
\end{figure}

\begin{figure}[h!]
\centering
\flushleft{$\ell = 0:$}\\
\includegraphics[width = 0.3\textwidth]{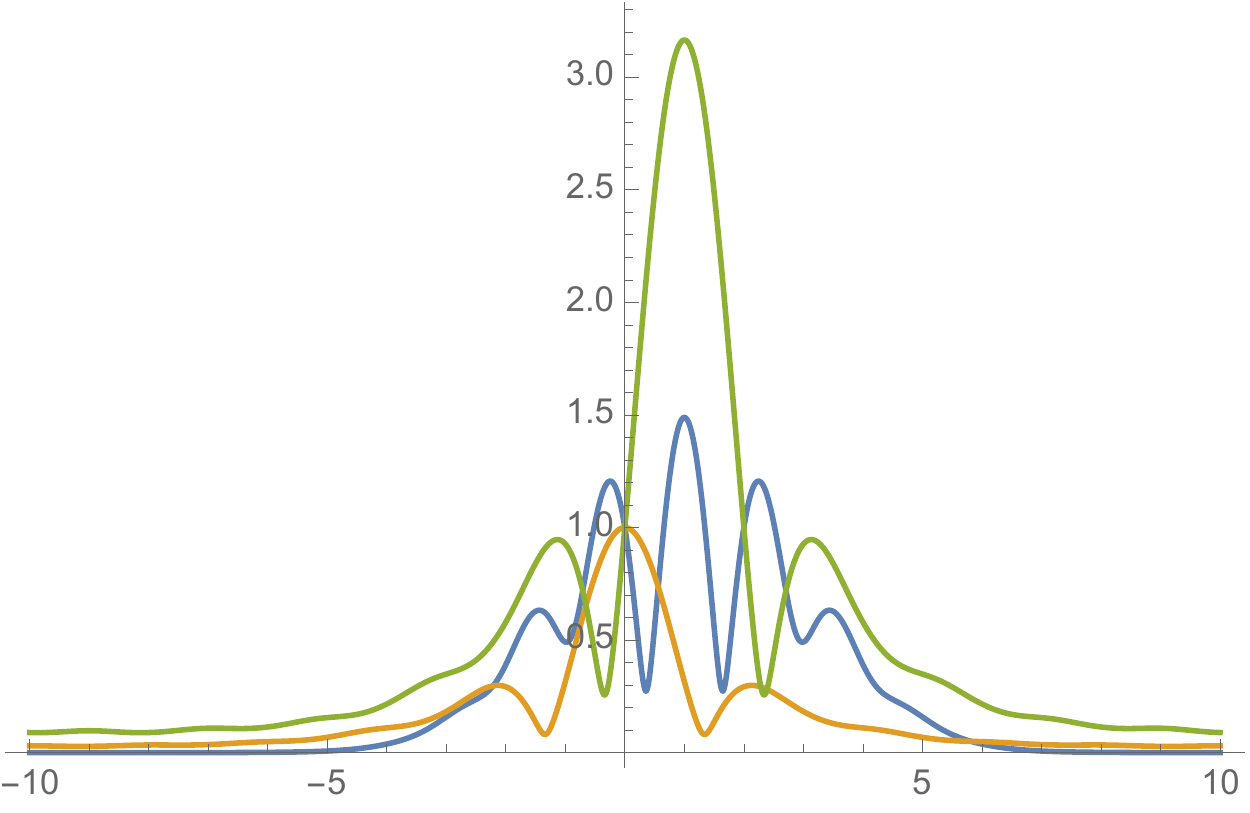}
\includegraphics[width = 0.3\textwidth]{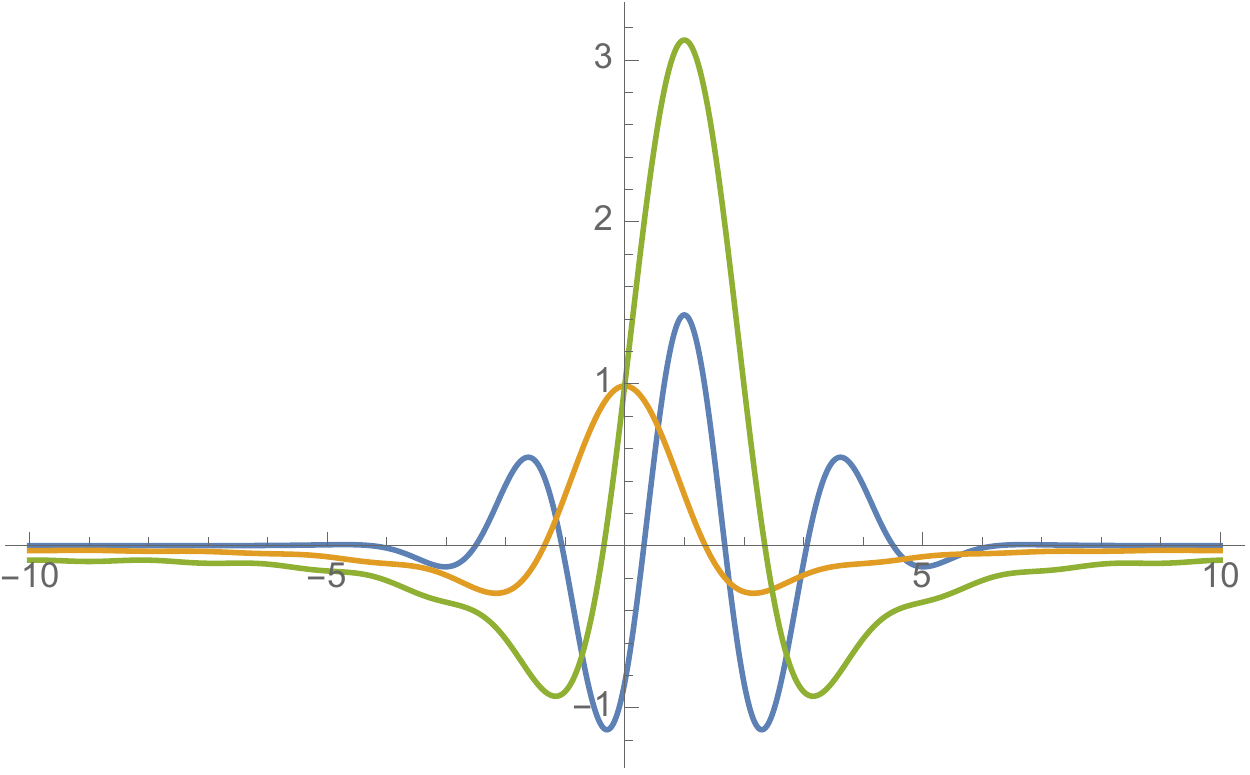}
\includegraphics[width = 0.3\textwidth]{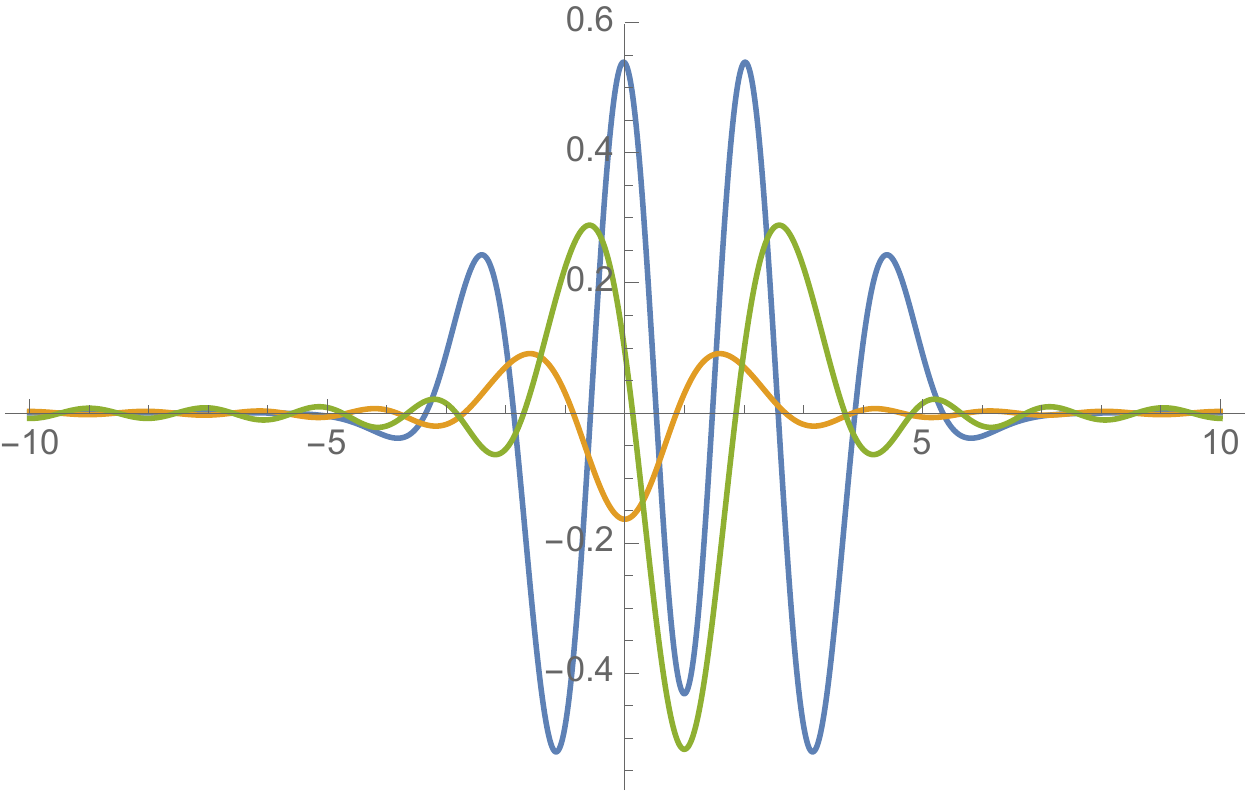}

\flushleft{$\ell = 1:$}\\
\includegraphics[width = 0.3\textwidth]{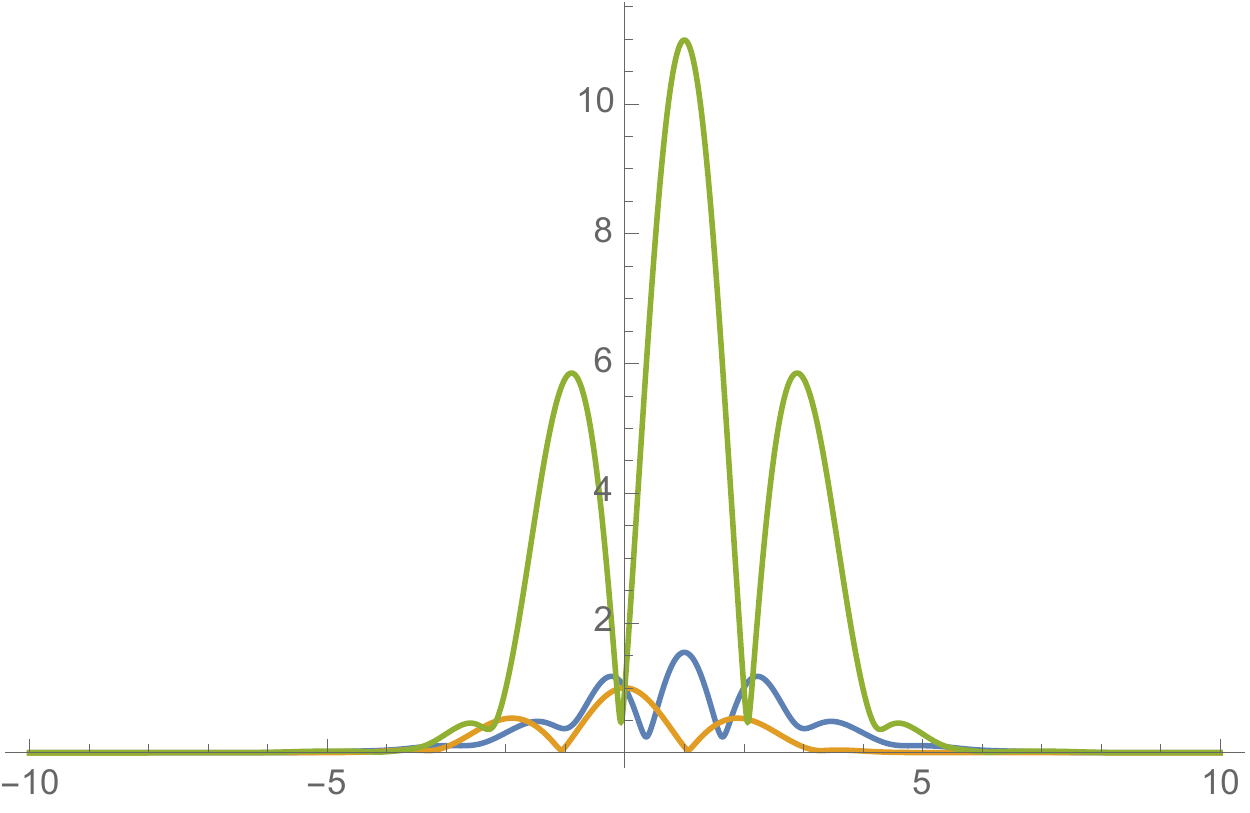}
\includegraphics[width = 0.3\textwidth]{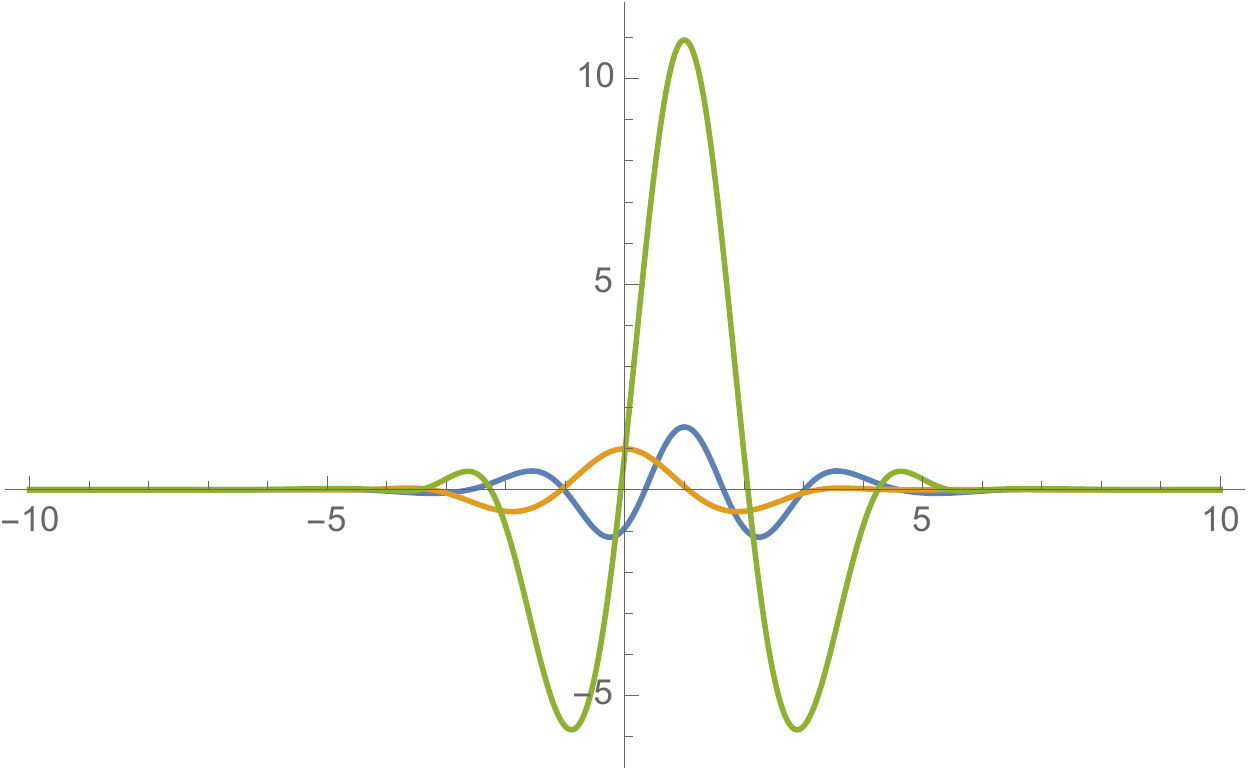}
\includegraphics[width = 0.3\textwidth]{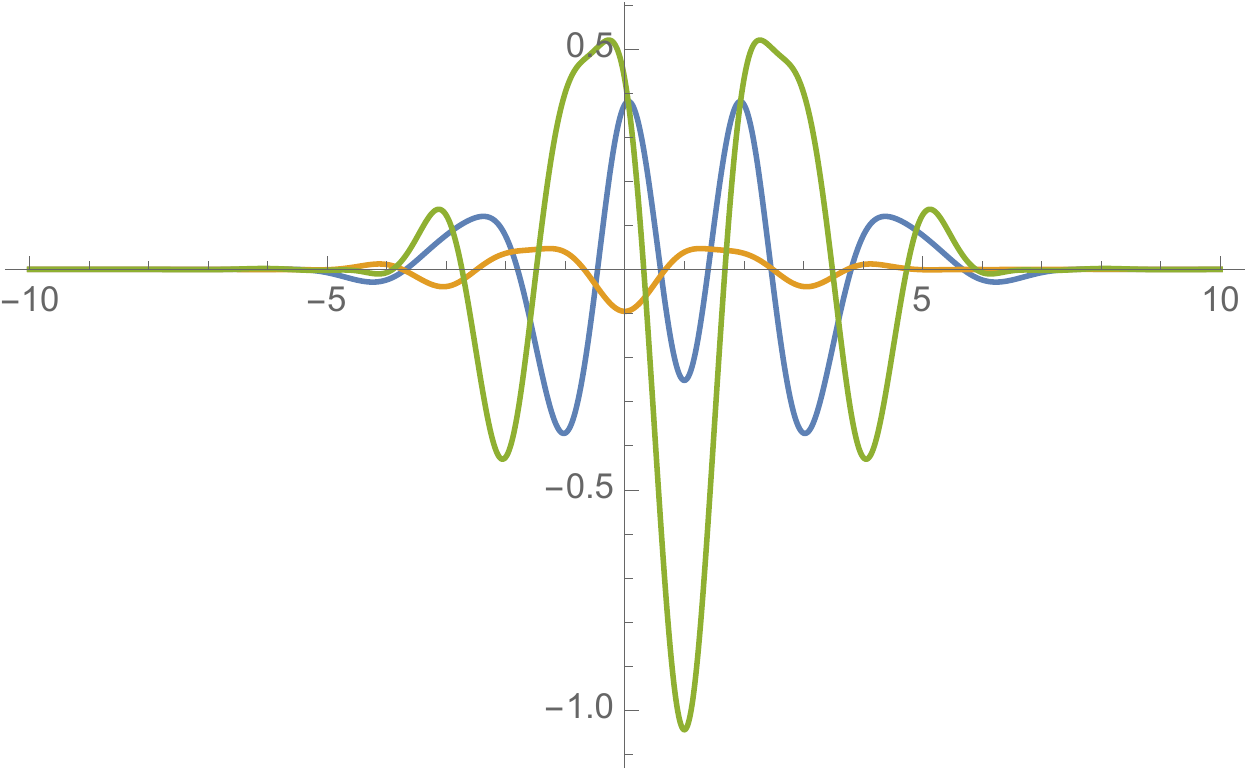}

\flushleft{$\ell = 2:$}\\
\includegraphics[width = 0.3\textwidth]{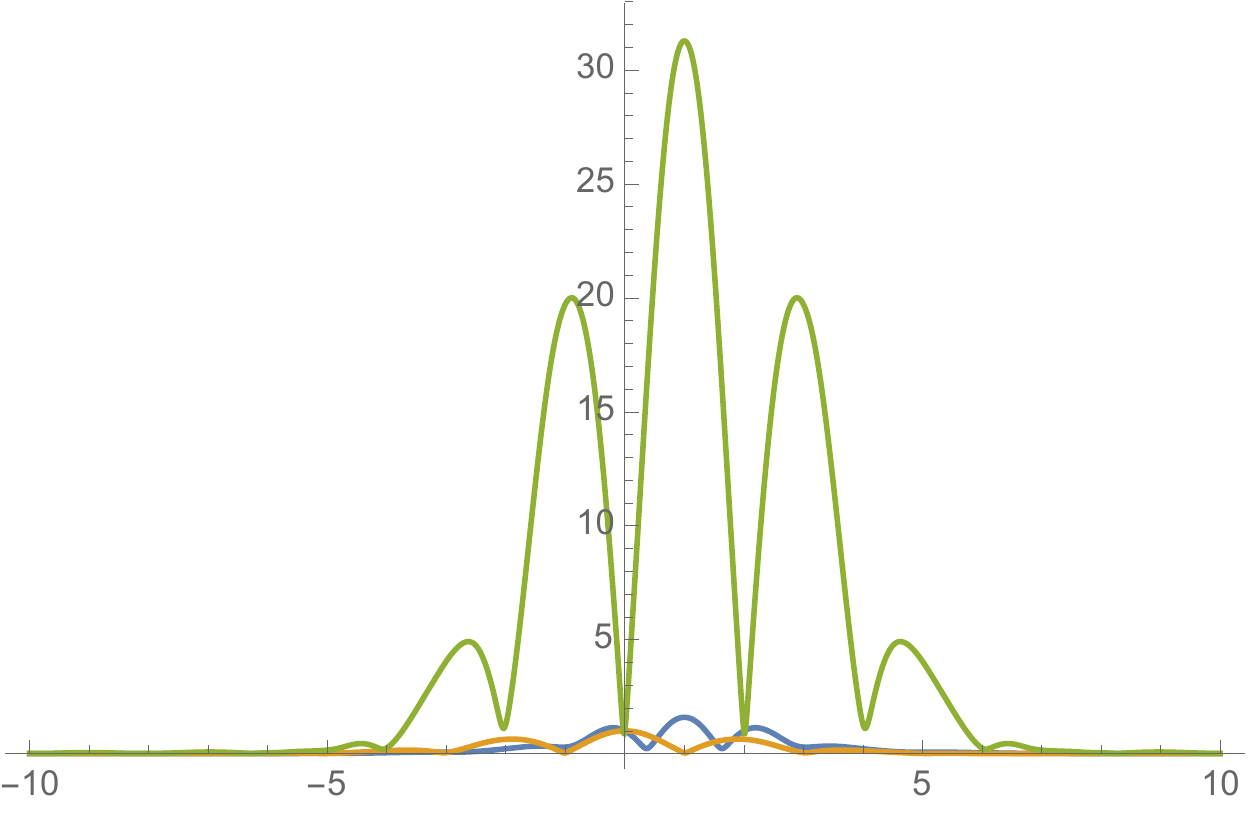}
\includegraphics[width = 0.3\textwidth]{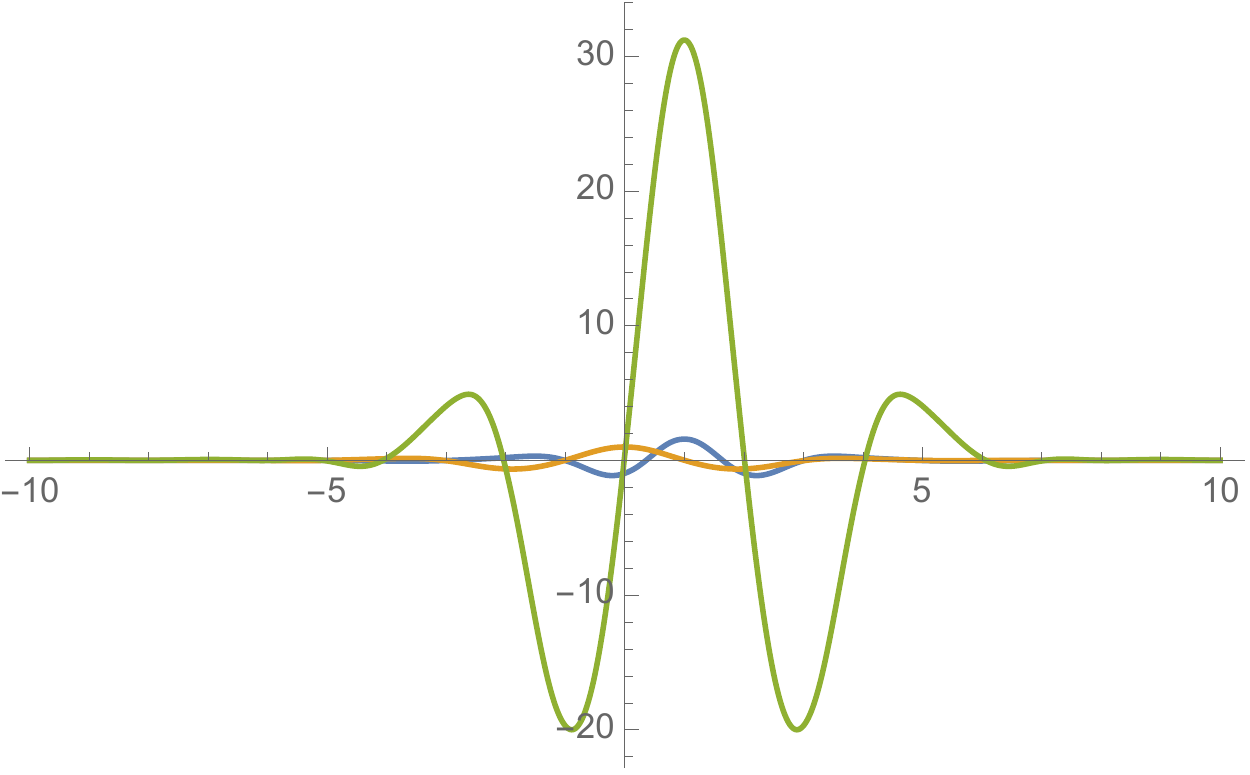}
\includegraphics[width = 0.3\textwidth]{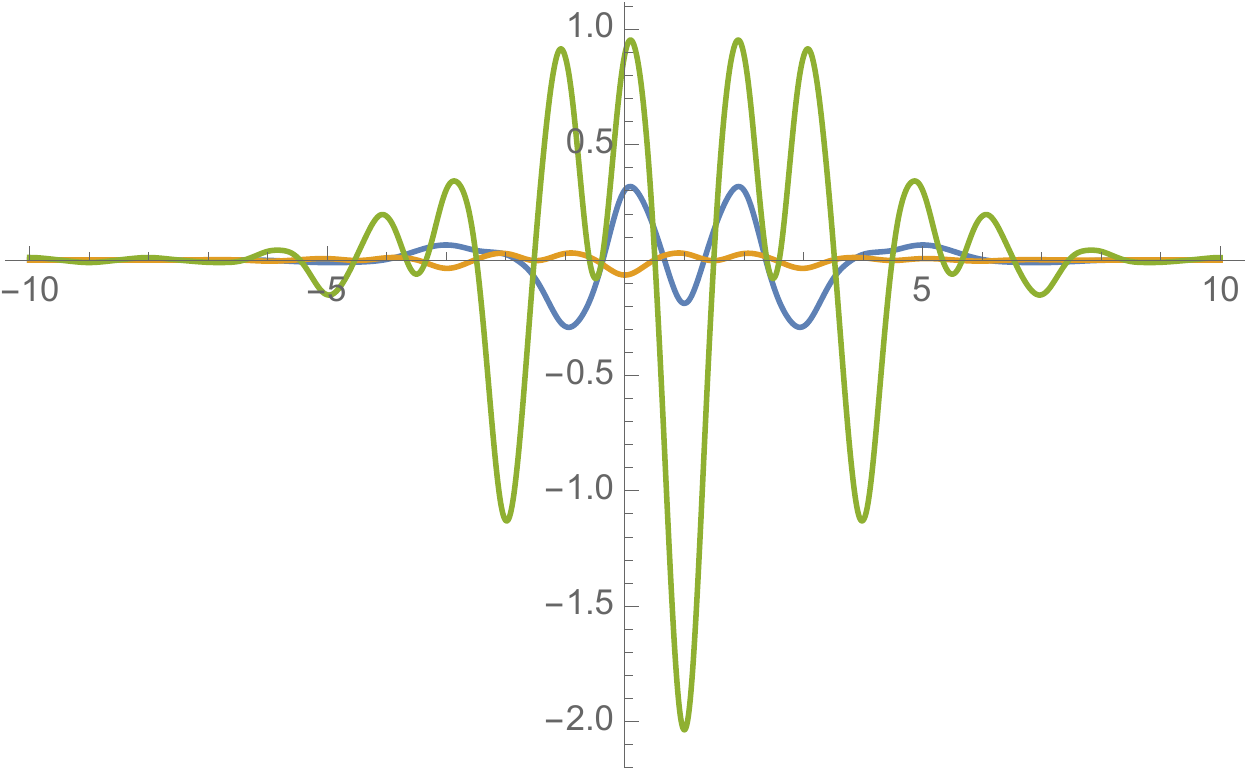}

\flushleft{$\ell = 3:$}\\
\includegraphics[width = 0.3\textwidth]{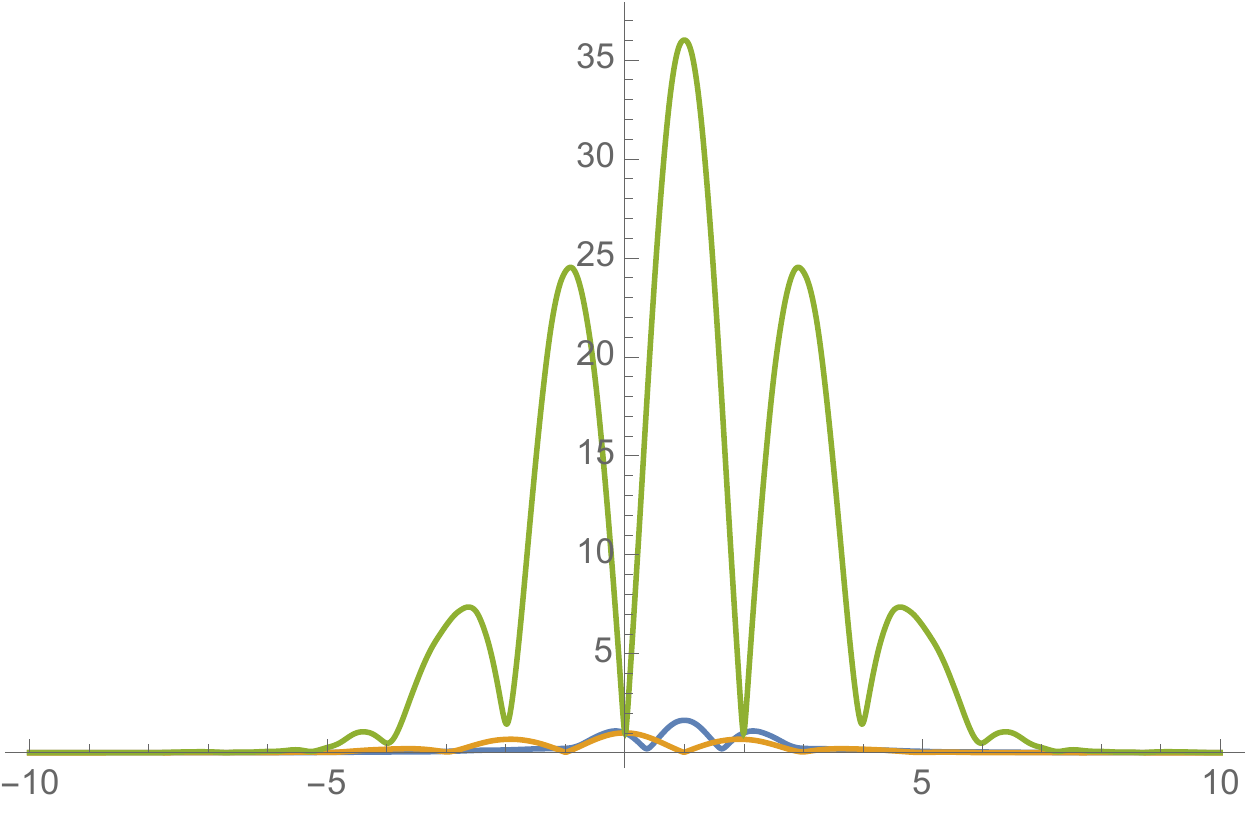}
\includegraphics[width = 0.3\textwidth]{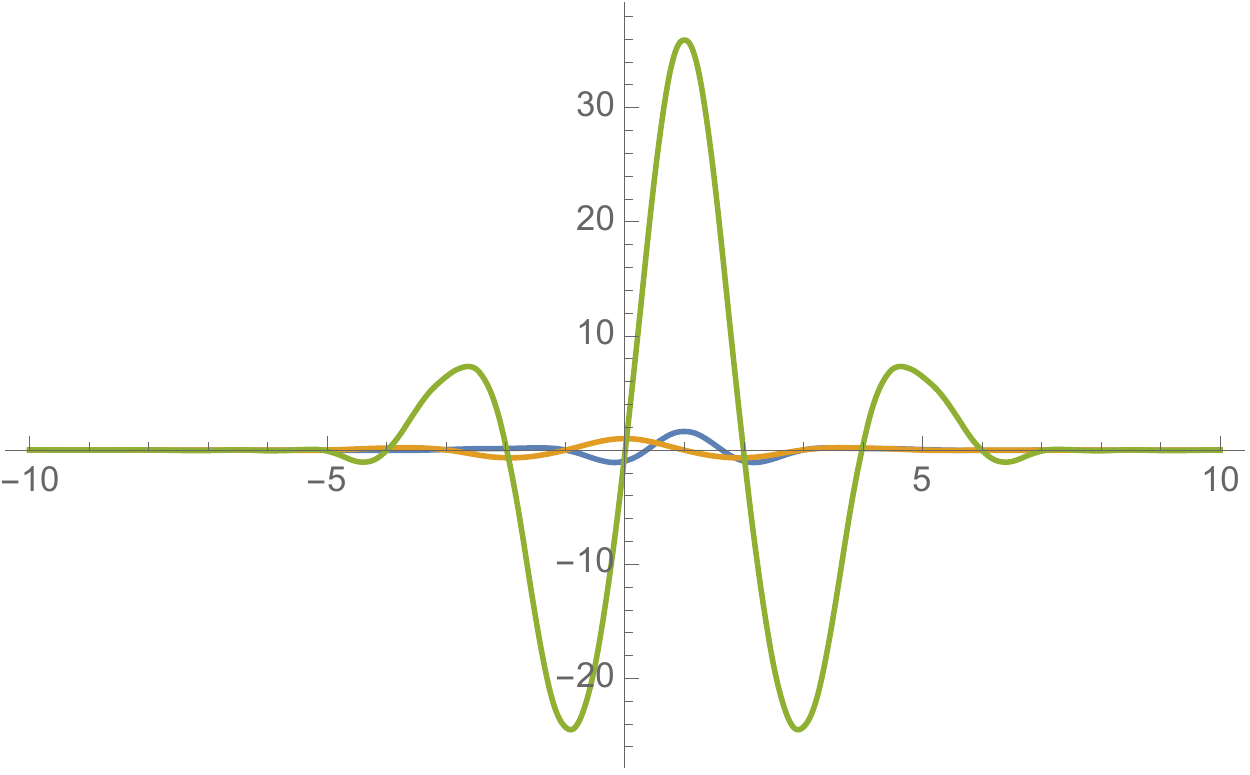}
\includegraphics[width = 0.3\textwidth]{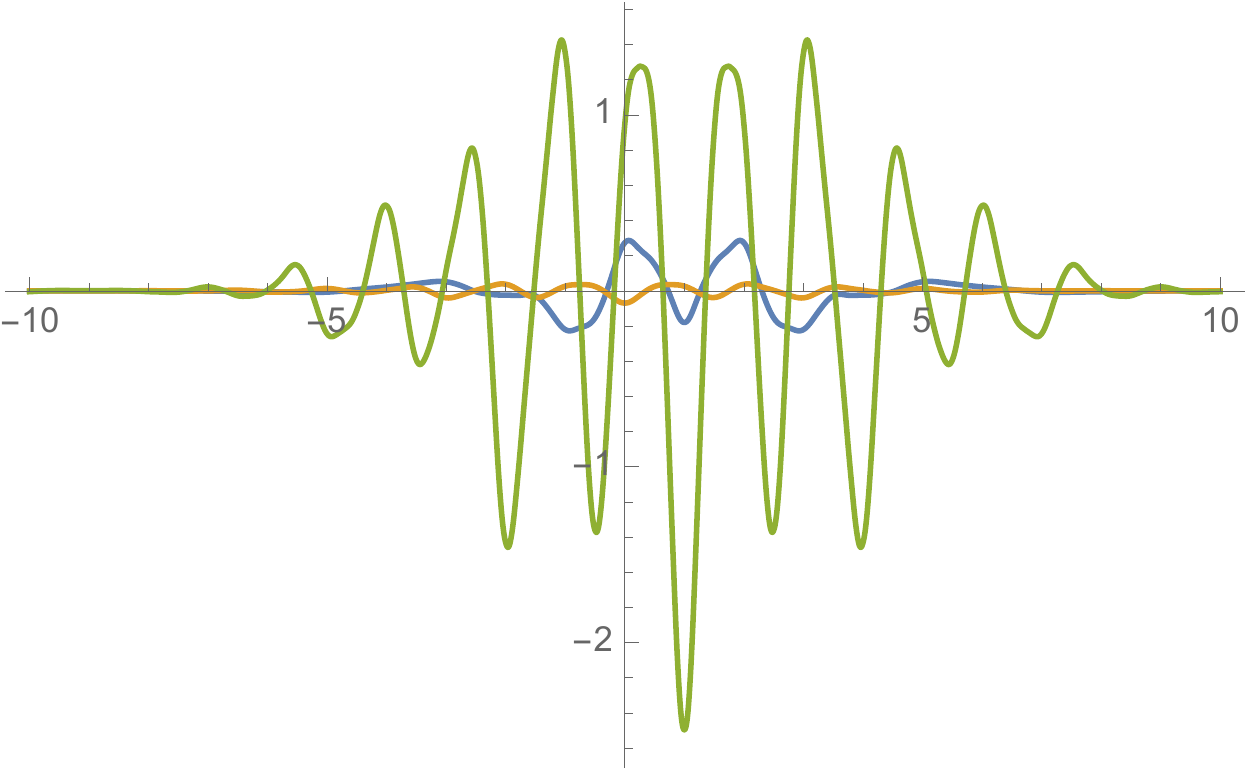}

\label{fig Wavelets im Zeit-Bereich}
\caption{The three framelets $\psi_1$,$\psi_2$, and $\psi_3$ from Proposition \ref{9603} in the time domain. Left: modulus, center: real part, right: imaginary part. Parameter: $z = 3.2 + i$.}
\end{figure}

\section{Regularity and Approximation Properties of Fractional Pseudo-Splines}
In this section, we investigate the regularity and approximation properties of fractional pseudo-splines. For this purpose, we recall that in the case $\ell = 0$, the fractional B-splines $\beta^{2\alpha - 1}$ as introduced in \cite{unserblu00} have (fractional) approximation order $2\alpha$.

The result about the regularity of fractional pseudo-splines is based on the following theorem, which is Lemma 7.1.7 in \cite{daubechies92}, adapted to our setting, terminology, and notation.
\begin{lemma}\label{thm5.1}
Let $H_0 (\gamma)$, the low pass filter for a refinable function $\varphi$, have the form
\[
|H_0(\gamma)| = (\cos^2\pi\gamma)^\alpha |\mathcal{L}(\gamma)|, \quad\gamma\in \T,
\]
where $\alpha\geq 1$. Assume that
\begin{enumerate}
\item[\emph{(i)}] $|\mathcal{L}(\gamma)|\leq |\mathcal{L}(\frac13)|$, for all $|\gamma|\leq \frac13$;
\item[\emph{(ii)}] $|\mathcal{L}(\gamma)\mathcal{L}(2\gamma)|\leq |\mathcal{L}(\frac13)|^2$, for all $\frac13\leq|\gamma|\leq\frac12$.
\end{enumerate}
Then
\[
|\widehat{\varphi} (\gamma)| \leq c\,(1 + |\gamma|)^{-2\alpha + \kappa},
\]
where $\kappa = \log_2 |\mathcal{L}(\frac13)|$ and $c$ denotes a positive constant.
\end{lemma}
\begin{proof}
The proof follows from the arguments given to establish Lemmata 7.1.5 and 7.1.6 in \cite{daubechies92}. We note that replacing the positive integer $N$ in \cite{daubechies92} by a real number $\alpha\geq 1$ has no bearing on the validity of these arguments.
\end{proof}

The next result is the extension of Proposition 3.2 in \cite{DongShen} to the fractional case.
\begin{lemma}
Suppose $p:= p_{\alpha, \ell} := \sum\limits_{k=0}^\ell \binom{\alpha + \ell}{k} x^k (1-x)^{\ell-k}$, where $\alpha\geq 1$ and $\ell = 0, 1, \ldots, \lfloor\alpha-\frac12\rfloor$. Then
\begin{enumerate}
\item[\emph{(a)}] $p(x) \leq p(\frac34)$, for $0\leq x \leq \frac34$;
\item[\emph{(b)}] $p(x) p(4x(1-x)) \leq p(\frac34)^2$, for $\frac34\leq x\leq 1$.
\end{enumerate}
\end{lemma}
\begin{proof}
For item (a), we remark that by \eqref{p} and the positivity of the binomial coefficients, we have
\[
p'(x) = \alpha + 2\binom{\alpha + 1}{2} x + \cdots + \ell \binom{\alpha-1+\ell}{\ell} x^{\ell-1} > 0,
\]
for all $x\in [0,1]$.
\ml
For (b), a careful analysis of the proof of Proposition 3.2 in \cite{DongShen} reveals that the arguments given there also apply when the parameter $m \in \N$ in the statement is replaced by a real number
$\alpha\geq 1$. Only the identity \eqref{iden} was used and the algebraic manipulations performed are independent of whether $m\in \N$ or $1\leq\alpha\in \R$.
\end{proof}
As a consequence, we obtain the following regularity result for fractional pseudo-splines.
\begin{thm}
Let $\varphi$ denote a pseudo-spline of order $(\alpha, \ell)$ with $1\leq\alpha\in \R$ and $\ell = 0,1, \ldots, \lfloor\alpha-\frac12\rfloor$. Then
\[
|\widehat{\varphi} (\gamma)| \leq c (1 + |\gamma|)^{-2\alpha + \kappa},
\]
where $\kappa = \log_2 p(\frac34)$. Moreover,
\[
\varphi\in C^{s-\varepsilon}, \quad\text{with}\quad s = 2\alpha - \kappa - 1.
\]
\end{thm}
\begin{proof}
The first statement follows from Theorem \ref{thm5.1}, the fact that in our setting, $\mathcal{L} (\gamma)$ is given by $P^{(\alpha, \ell)} (\gamma)$ (see \eqref{Pzl} with $z$ replaced by $\alpha$), and the observation that the proof of Theorem 3.4 in \cite{DongShen} also applies when the
parameter $m\in \N$ used in the statement is replaced
by any real $m\geq 1$. The second statement follows from the fact that if $ |\widehat{f} (\gamma)| \leq c (1+|\gamma|)^{-s - 1 - \varepsilon}$, $c > 0$, then $\int_\R |\widehat{f} (\gamma)| (1+|\gamma|)^s d\gamma < \infty$, i.e., $f\in C^s$.
\end{proof}

Next, we consider the approximation order associated with a fractional pseudo-spline. To this end, let $\varphi$ be a pseudo-spline of order $(\alpha, \ell)$ with $1\leq\alpha\in \R$ and $\ell = 0,1, \ldots, \lfloor\alpha-\frac12\rfloor$. Define shift-invariant subspaces $\{V_n\}_{n\in \Z}$, $n\in \Z$, of $L^2(\R)$ by
\[
V_n := \clos_{L^2} \Span \left\{D^n T_k\varphi : k\in \Z\right\}
\]
and denote by $\sP_n: L^2(\R)\to V_n$ the linear operator
\be\label{eq16}
\sP_n(f) := \sum_{k\in \Z} \inn{f}{D^n T_k\varphi}\,D^n T_k\varphi.
\ee
The operator $\sP_n$ is said to provide \emph{approximation order $\alpha$} if
\[
\Vert f - \sP_n f \Vert_{L^2} \in \mathcal{O}(2^{-n\alpha})\quad\text{as $n\to\infty$},
\]
for all $f$ in the Sobolev space $H^\alpha (\R)$.

We require the following known result adapted to our setting. For a proof and details, see for instance \cite{DRoSh3}.
\begin{lemma}
The approximation order of the operator $\sP_n$ is given by $\min\{2\alpha, M\}$, where $M$ is the order of the zero of $1 - |H_0^{(\alpha,\ell)}(\gamma)|^2$ at the origin.
\end{lemma}

We then arrive at the next theorem.
\begin{proposition}
Let $\varphi$ be a pseudo-spline of order $(\alpha, \ell)$ with $1\leq\alpha\in \R$ and $\ell = 0,1, \ldots, \lfloor\alpha-\frac12\rfloor$. The operator $\sP_n$ defined in \eqref{eq16} provides approximation order $\min\{2\alpha, 2(\ell+1)\}$.
\end{proposition}
\begin{proof}
Note that
\[
1 - |H_0^{(\alpha,\ell)}(\gamma)|^2 = 1 - q^2(\sin^2\pi\gamma),
\]
where $q$ is defined in Lemma \ref{60116a}. The result now follows from the proof of Theorem 3.10 in \cite{DongShen} with the obvious replacements and the realization that their arguments remain unchanged when the integer $m$ in the statement is replaced by the real number $\alpha$.
\end{proof}
\begin{rem}
The above statements concerning the regularity and approximation order of the refinable functions generated by the filters $H_0^{(\alpha,\ell)}$ apply ad verbatim also to the shifted filters $H_0^{(\alpha,\ell,u)}$ since their moduli are identical.
\end{rem}
\bibliographystyle{plain}
\bibliography{Pseudosplines.bib}

\begin{thebibliography}{10}

\bibitem{C}
O.~Christensen.
\newblock {\em An Introduction to Frames and Riesz Bases}.
\newblock Birkh{\"a}user, Bosten, 2nd. edition, 2016.

\bibitem{daubechies92}
I.~Daubechies.
\newblock {\em Ten Lectures on Wavelets}.
\newblock CBMS-NSF Regional Conference Series in Applied Mathematics. Society
  for Industrial and Applied Mathematics, 1992.

\bibitem{DRoSh3}
I.~Daubechies, B.~Han, A.~Ron, and Z.~Shen.
\newblock Framelets: {M}{R}{A}-based construction of wavelet frames.
\newblock {\em Appl. Comput. Harmon. Anal.}, 14(1):1--46, 2003.

\bibitem{DongShen}
B.~Dong and Z.~Shen.
\newblock Pseudo-splines, wavelets and framelets.
\newblock {\em Appl. Comp. Harmon. Anal.}, 22:78--104, 2007.

\bibitem{DongShen2}
B.~Dong and Z.~Shen.
\newblock {MRA}-based wavelet frames and applications.
\newblock In {\em The Mathematics of Image Processing}, volume~19 of {\em
  IAS/Park City Mathematics Series}, pages 7--158. American Mathematical
  Society, 2010.

\bibitem{forster14}
B.~Forster.
\newblock Five good reasons for complex-valued transforms in image processing.
\newblock In A.~I. Zayed and G.~Schmeisser, editors, {\em New Perspectives on
  Approximation and Sampling Theory}, Applied and Numerical Harmonic Analysis,
  chapter~15, pages 359--381. Birkh\"auser, 2014.

\bibitem{forster06}
B.~Forster, T.~Blu, and M.~Unser.
\newblock Complex {B}-splines.
\newblock {\em Appl. Comp. Harmon. Anal.}, 20:281--282, 2006.

\bibitem{lmr-englisch}
A.~K. Louis, P.~Maass, and A.~Rieder.
\newblock {\em Wavelets: Theory and Applications}.
\newblock John Wiley \& Sons, 1997.

\bibitem{luisier}
F.~Luisier, T.~Blu, B.~Forster, and M.~Unser.
\newblock Which wavelet bases are the best for image denoising?
\newblock In {\em Proceedings of the SPIE Optics and Photonics 2005 Conference
  on Mathematical Methods: Optical Engineering and Instrumentation (Wavelet
  XI), San Diego CA, USA}, volume 5914, 2005.

\bibitem{RoSh2}
A.~Ron and Z.~Shen.
\newblock Affine systems in ${L}_2 (\mathbb{{R}}^d)$: the analysis of the
  analysis operator.
\newblock {\em J. Funct. Anal.}, 148(2):408--447, 1997.

\bibitem{RoSh4}
A.~Ron and Z.~Zhen.
\newblock Compactly supported tight affine spline frames in
  ${L}_2(\mathbb{R}^d)$.
\newblock {\em Math. Comp.}, 67:191--197, 1998.

\bibitem{unserblu00}
M.~Unser and T.~Blu.
\newblock Fractional splines and wavelets.
\newblock {\em {SIAM} Review}, 42(1):43--67, March 2000.

\bibitem{woj97}
P.~Wojtaszczyk.
\newblock {\em A Mathematical Introduction to Wavelets}.
\newblock Number~37 in London Mathematical Society Student Texts. Cambridge
  University Press, 1997.

\bibitem{Zhuang}
Z.~Zhuang and J.~Yang.
\newblock A class of generalized pseudo-splines.
\newblock {\em Journal of Inequaltities and Applications}, (359):1--10, 2014.

\end{thebibliography}

\vspace{.1in}\noindent Peter Massopust\\
Zentrum Mathematik, M15\\
Technische Universit\"at M\"unchen\\
Boltzmannstr. 3\\
85747 Garching, Germany\\
massopust@ma.tum.de

\vspace{.1in}\noindent Brigitte Forster\\
Fakult\"at f\"ur Informatik und Mathematik\\
Universit\"at Passau\\
Innstr. 33\\
94032 Passau, Germany\\
brigitte.forster@uni-passau.de
{\vspace{.1in}
\noindent 

Ole Christensen\\
Department of Applied Mathematics and Computer Science\\
Technical University of Denmark\\
Building 303\\
2800 Lyngby,
Denmark \\
ochr@dtu.dk
}

\end{document}